\documentclass[12pt]{amsart}
\usepackage[english]{babel}
\usepackage{paralist,url,verbatim, anysize}
\usepackage{amscd}
\usepackage[all,cmtip]{xy} 
\usepackage{amsmath,amsthm,amssymb,amsfonts}
\usepackage[bookmarksopen=true]{hyperref}
\usepackage[usenames, dvipsnames]{color}
\usepackage{graphicx}
\usepackage{hyperref}
\usepackage{euler, amsfonts, amssymb, latexsym, epsfig,epic}
\usepackage{stmaryrd}

\usepackage{tikz}
\usetikzlibrary{decorations.markings,intersections,positioning,calc}

  \tikzset{mylabel/.style  args={at #1 #2  with #3}{
    postaction={decorate,
    decoration={
      markings,
      mark= at position #1
      with  \node [#2] {#3};
 } } } }

\newcommand{\ps}[1]{\llbracket {#1} \rrbracket}
\newcommand{\pps}[1]{\llparenthesis {#1} \rrparenthesis}

\DeclareMathOperator{\ann}{ann}
\makeatletter
\let\@fnsymbol\@arabic
\makeatother

\theoremstyle{plain}
\newtheorem{theorem}{\bf Theorem}[section]

\newtheorem{corollary}[theorem]{Corollary}
\newtheorem{lemma}[theorem]{Lemma}
\newtheorem{proposition}[theorem]{Proposition}

\newtheorem{theoremx}{Theorem}
\newtheorem{examplex}[theoremx]{Example}

\theoremstyle{definition}

\newtheorem{definition}[theorem]{Definition}

\newtheorem{example}[theorem]{\bf Example}
\newtheorem{remark}[theorem]{Remark}

\newtheorem*{theorem*}{\bf Theorem}

\newcommand{\codim}{\operatorname{codim} }

\newcommand{\init}{\operatorname{in} }
\newcommand{\chara}{\operatorname{char} }

\newcommand{\N}{\mathbb{N}}
\newcommand{\Z}{\mathbb{Z}}

\newcommand{\asdim}{\operatorname{a-sdim} }
\newcommand{\gsdim}{\operatorname{sdim} }
\newcommand{\Min}{\operatorname{Min} }
\newcommand{\Ass}{\operatorname{Ass} }
\newcommand{\Tor}{\operatorname{Tor} }
\newcommand{\height}{\operatorname{height} }
\newcommand{\Ker}{\operatorname{Ker} }

\renewcommand{\Im}{\operatorname{Im} }
\newcommand{\Proj}{\operatorname{Proj} }
\newcommand{\Spec}{\operatorname{Spec} }
\newcommand{\Ext}{\mathrm{Ext}}
\newcommand{\depth}{\operatorname{depth} }

\newcommand{\pp}{\mathfrak{p}}
\newcommand{\qq}{\mathfrak{q}}

\newcommand{\mm}{\mathfrak{m}}
\newcommand{\aaa}{\mathfrak{a}}

\newcommand{\nn}{\mathfrak{n}}

\newcommand{\homo}{\operatorname{hom} }

\definecolor{mypink}{RGB}{215, 5, 234}

\newcommand{\ZZ}{\ensuremath{\mathbb{Z}}}

\newcommand{\e}{\operatorname{e}}
\newcommand{\HP}{\mathrm{HP}}
\newcommand{\HS}{\mathrm{HS}}
\newcommand{\HF}{\mathrm{HF}}
\newcommand{\whomo}{\overline{\operatorname{hom}}}

\begin{document}
\title{From a local ring to its associated graded algebra}
\author[A. De Stefani]{Alessandro De Stefani} 
\email{alessandro.destefani@unige.it}
\author[M.E. Rossi]{Maria Evelina Rossi}
\email{rossim@dima.unige.it}
\author[M. Varbaro]{Matteo Varbaro} 
\email{matteo.varbaro@unige.it}
\address{Dipartimento di Matematica, Universit\'a di Genova, Italy} 
  \date{}
\maketitle

\begin{abstract}
Let $(R,\mathfrak{m})$ be a complete local ring, and $G={\rm gr}_{\mathfrak{m}}(R)$ be its associated graded ring. We introduce a homogenization technique which allows to relate $G$ to the special fiber and $R$ to the generic fiber of a ``Gr{\"o}bner-like" deformation. Using this technique we prove sharp results concerning the connectedness of $R$ and $G$. We also construct a family of local domains which fail to satisfy Abhyankar's inequality for the Hilbert-Samuel multiplicity. However, we prove a version of the inequality which holds when $R$ is connected in codimension one.
\end{abstract}

\section{Introduction}

Let $(R,\mm,k)$ be a local ring, and $G=\bigoplus_{i\in\N}\mm^i/\mm^{i+1}$ be its associated graded $k$-algebra. In this article we introduce a technique which is inspired by Gr\"obner deformations, and which plays a fundamental role in relating properties of $R$ and $G$. More specifically, we associate $G$ to the special fiber and $R$ to the generic fiber of a deformation with total space which we call $\homo(R)$. For our purposes, it is crucial that $\homo(R)$ is a complete local ring.


\begin{theoremx}[see Theorems \ref{t:deform equichar} and \ref{t:deform mixed}] \label{THMX 0}
Let $(R,\mm,k)$ be a complete local ring. There exists a complete local ring $\homo(R)$, and non-zero divisor $t$ on $\homo(R)$, such that 
\[
\homo(R)/(t)\cong \widehat{G} \ \text{ and } \ \widehat{\homo(R)\otimes_{V\ps{t}} V\pps{t}}\cong \widehat{R\otimes_V V\pps{t}},
\]
where $V=k$ if $R$ has equal characteristic, while it is a complete unramified discrete valuation ring with residue field $k$ if $R$ has mixed characteristic.
\end{theoremx}

Theorem \ref{THMX 0} has several advantages compared to the more classical approaches, such as the one using the extended Rees algebra. Among others, $\homo(R)$ is a complete local ring, as already mentioned. Moreover, it has a rather concrete description in terms of defining equations.

In order to obtain Theorem \ref{THMX 0} we run into some technical difficulties. For example, the proof that the completion of $R$ is a domain if and only if $\homo(R)$ is a domain (see Proposition \ref{p:prime general}) is not as smooth as the analog in the framework of classical Gr\"obner bases theory. Moreover, the equal and mixed characteristic cases need a different treatment, although both approaches include the case when the ring $R$ has positive characteristic. In this situation, they turn out to agree (see Remark \ref{r:agree}). In classical Gr\"obner bases theory, this type of construction has proved itself to be extremely powerful in many regards (see \cite{DEP,BaSt,KalkSturmfels,Green,KnMi,ConcaVarbaro} for just a small selection). We believe that the one introduced in this article for local rings could have similar useful applications, and part of this work is committed to studying some of them.

While classical results in the theory of local rings ensure that several desirable properties such as being reduced, normal, Cohen-Macaulay, Gorenstein, complete intersection and many others carry over from $G$ to $R$ (for instance, see \cite{CavaliereNiesi,AvramovAchilles,Froberg}), the construction of Theorem \ref{THMX 0} allows to get information in the other direction. Namely, it allows to study some properties that pass from $R$ to $G$. 
For example, it is well known that if $R$ is Cohen-Macaulay then $G$ needs not be Cohen-Macaulay. However, we prove that in this case $G$ is at least connected in every positive codimension (see Section \ref{sconnsub} for the definition of ``connected in codimension $s\in\N$"). 
More precisely we show: 

\begin{theoremx}[see Theorem \ref{t:connectedness}] \label{THMX A} Let $R$ be a complete local ring and $G$ be its associated graded ring. Let $s>0$ be an integer. If $R$ is connected in codimension $s$, then $G$ is connected in codimension $s$. If $G$ is reduced, then the converse holds true as well.
\end{theoremx}

Theorem \ref{THMX A} is the analog of certain results in classical Gr\"obner bases theory, which were inspired by a conjecture of Kredel and Weispfenning \cite{KW}, and were obtained in various steps in \cite{KalkSturmfels, Varbaro, ALNBRM}. We point out that a related result to Theorem \ref{THMX A} was obtained by Brodmann \cite[Proposition (2.5)]{Brodmann} in terms of connectedness dimension of the blowup $\pi: \Proj(R[I T]) \to \Spec(R)$ and its special fiber $\pi^{-1}(V(J))$ for two given ideals $I \subseteq J \subseteq \mm$,
using results of \cite{BrodmannRung}. Translated to our language, in the case $I=J=\mm$ Brodmann's theorem states that if $R$ is connected in codimension $s \geq 0$, then $G$ is connected in codimension $s+1$. Note that, if $R$ is connected in codimension $s=0$, then it is also connected in codimension $s=1$. Therefore the first part of Theorem \ref{THMX A} strengthens the case $I=J=\mm$ of \cite[Proposition (2.5)]{Brodmann} for all $s \geq 0$. 
We also obtain results relating the subdimensions of $R$ and $G$, proving that $\gsdim(G) \geq \gsdim(R)$ always holds under mild assumptions on $R$. Moreover, equality holds if $G$ satisfies Serre's condition $(S_1)$, see Theorem \ref{t:algsdim}.

In Section \ref{Section bounds}, we focus on numerical invariants of local rings. Among these, we consider the Hilbert-Samuel multiplicity $\e(R) = \lim_{n \to \infty} d! \ell_R(R/\mm^n)/n^d$, where $d=\dim(R)$. Recall that $\e(R)$ coincides with the multiplicity of its associated graded ring $G$, that is, the (normalized) leading coefficient of its Hilbert polynomial. When $(R,\mm)$ is Cohen-Macaulay, a celebrated inequality of Abhyankar gives that $\e(R) \geq \mu(\mm) - \dim(R) + 1$, where $\mu(-)$ denotes the minimal number of generators of a module \cite{Abhyankar}. 
For Cohen-Macaulay standard graded $k$-algebras Abhyankar's inequality is trivial, and it is true also for standard graded $k$-algebras connected in codimension 1 assuming that $k$ is an algebraically closed field (see for instance \cite[Proposition 5.2]{GSyz}). One might wonder if this is still true for local domains; however, this is not the case.

%

\begin{examplex}[Example \ref{Ex}] \label{THMX B} Let $k$ be a field. There exists a family of $2$-dimensional complete local domains $\{(R_n,\mm_n)\}_{n \in \ZZ_{>0}}$, containing $k$, with $\mu(\mm_n)=4$ and $\e(R_n) = 2$ for all $n$. Furthermore, each $R_n$ is an isolated singularity and its associated graded ring is unmixed.
\end{examplex}

For every $n$, the associated graded ring of each above local domain $R_n$ has two nilpotent linear forms which are linearly independent over $k$. If one takes into account the $k$-dimension of the space of linear forms that are nilpotent in the associated graded ring, one gets a version of Abhyankar's inequality which holds for every local ring which is connected in codimension one.

\begin{theoremx}[Theorem \ref{thm h2 and multiplicity linear forms}] Let $(R,\mm)$ be a complete local ring connected in codimension one, and assume that $R$ has an algebraically closed residue field. If $\ell=\dim_k\left(\left[G_{\mathrm{red}}\right]_1\right)$, then $\e(R) \geq \mu(\mm)-\dim R + 1 - \ell$. Moreover, if $G$ satisfies Serre's condition $(S_1)$ and $\ell \ne 0$, then the inequality is strict.
\end{theoremx}

\section{Gr{\"o}bner deformations for complete local rings} \label{Section deformation}

In this section we develop the main tool of this article: a ``Gr{\"o}bner-deformation'' argument which allows to relate properties of any complete local ring $(R,\mm)$ with those of its associated graded ring $G$. The treatment will be slightly different for rings that contain a field and rings of mixed characteristic: for the latter, we need to employ a trick which allows to connect both $G$ and $R$ to a quotient of a single complete unramified regular local ring of mixed characteristic. 

First we prove some results which are common to both setups. The following notation will be adopted throughout this section: let $(A,\mm_A,\kappa_A)$ be a complete regular local ring which either contains a field or is unramified, that is, $\chara(A) \ne \chara(\kappa_A) \notin \mm_A^2$. Let $S=A\ps{x_1,\ldots,x_n}$, $\mm=(x_1,\ldots,x_n)$, and consider $\psi:S\to S\ps{t}=:T$ the homomorphism of $A$-algebras defined by setting $x_i \mapsto x_it$. Given $0 \ne f \in S$ we let $o(f) = \sup\{i \mid f \in \mm^i\}$, and we set $\homo(f) = t^{-o(f)}\psi(f)$. It follows immediately from the given definition that $\homo(fg) = \homo(f)\homo(g)$ for all $0\ne f,g \in S$. If $I \subseteq S$ is an ideal we let 
\[
\homo(I) = (\homo(f) \mid 0\ne f \in I) \subseteq T.
\]

Note that, by the very definition of $T \cong A\ps{x_1,\ldots,x_n}\ps{t}$, any element $F \in T$ can be written uniquely as $\sum_{i \geq 0}t^i \sum_{j \geq 0} F_{(i,j)}$ for some polynomials $F_{(i,j)} \in A[x_1,\ldots,x_n]=:P$ which are homogeneous of degree $j$ in the variables $x_1,\ldots,x_n$ (all the variables of $P$ are given degree $1$, while any element of $A$ is given degree $0$).

\begin{definition} We say that $0 \ne F \in T$ is $t$-homogeneous of degree $d \in \Z$ if, when written as $F=\sum_{i \geq 0}t^i \sum_{j \geq 0} F_{(i,j)}$, one has $F_{(i,j)} = 0$ for all $j \ne i+d$.
\end{definition}

We observe that, given $0 \ne f \in S$, the element $\homo(f) \in T$ defined above is $t$-homogeneous of degree $o(f)$ by construction. In particular, $\homo(I)$ is a $t$-homogeneous ideal, that is, an ideal of $T$ generated by $t$-homogeneous elements.

\begin{lemma}\label{l:homogeneous general} Any $F\in T$ can be written in a unique way as $F=\sum_{d\in\Z}F_d$  where $F_d$ is $t$-homogeneous of degree $d$. If $J\subseteq T$ is a $t$-homogeneous ideal, then the following are equivalent:
\begin{enumerate} 
\item $F\in J$.
\item $F_d\in J$ for all $d\in\Z$.
\end{enumerate}
\end{lemma}
\begin{proof}
For the first part, write $F=\sum_{i \in \N }t^i \sum_{j \in \N} F_{(i,j)}$ with $F_{(i,j)}$ in $P$ homogeneous of degree $j$. 
For all $d \in \ZZ$ we set $F_d=\sum_{i \in \N} t^i F_{(i,i+d)}$, which is a $t$-homogeneous element of degree $d$ of $T$. The fact that the writing $F=\sum_{d\in\Z}F_d$ is unique is an easy exercise.

Concerning the last part of the statement, first assume $(1)$, and let $F_1,\ldots ,F_r$ be $t$-homogeneous generators of $J$ of degrees $d_1,\ldots,d_r \in \Z$. If $F\in J$ write
\[F=\sum_{i=1}^rG_iF_i, \ \ \ \text{ with } G_i\in T.\]
By the previous part each $G_i$ can be written as $G_i=\sum_{d\in\Z}G_{i,d}$ where each $G_{i,d}$ is $t$-homogeneous of degree $d$. It follows from the uniqueness of the writing that
\[
F_d=\sum_{i=1}^rG_{i,d-d_i}F_i \ \ \ \text{ for all } d \in \Z.
\]
In particular, $F_d\in J$ for all $d\in\Z$.

For the converse implication, assume that $F_d\in J$ for all $d\in\Z$. Then $G_d=\sum_{-d\leq j \leq d}F_j$ belongs to $J$ for all $d\in\N$. Notice that $(G_d)_{d\in\N}$ is a Cauchy sequence with respect to the $(x_1,\ldots ,x_n,t)$-adic topology on $T$, convergent to $F$. Since any ideal of $T$ is complete with respect to such topology, we conclude that $F\in J$.
\end{proof}

\begin{lemma} \label{l:nzd} Let $S$ and $T=S\ps{t}$ be as above, and $I \subseteq S$ be a proper ideal. Then $t$ is a non-zero divisor on $T/\homo(I)$.
\end{lemma}
\begin{proof}
Let $F\in T$ be such that $tF\in\homo(I)$, so that $tF=\sum_{i=1}^rG_i\homo(f_i)$ for some $f_i\in I$ and  $G_i\in T$. First, we prove the result further assuming that $F$ is $t$-homogeneous of a given degree $d \in \Z$. Note that, in this case, each $G_i$ is forced to be $t$-homogeneous of degree $d-o(f_i)$. For every $i$ we can thus write $G_i=\sum_{j\in\N}t^j g_{i,j}$, where $g_{i,j}\in P=A[x_1,\ldots,x_n]$ is a homogeneous polynomial of degree $d-o(f_i)+j$. Note that we necessarily have $\sum_{i=1}^rg_{i0}\homo(f_i) \in (t)T$. That is, there must exist $G\in T$ such that
\[\sum_{i=1}^rg_{i0}\homo(f_i)=tG.\]
If $G =0$, then clearly $G \in \homo(I)$. If $G \ne 0$, we write $G=t^sG'$ with $G' \notin (t)T$. If we set $g:=\sum_{i=1}^rg_{i0}f_i$, then $\homo(g)=G' \in \homo(I)$. In particular, we have that $G\in\homo(I)$ in this case as well. Now we can write
\[tF=\sum_{i=1}^r(G_i-g_{i0})\homo(f_i)+tG=t\left(\sum_{i=1}^r\left(\sum_{j\geq 1}g_{ij}t^{j-1}\right)\homo(f_i)+G\right).\]
Since $T$ is a domain, we obtain that $F=\sum_{i=1}^r(\sum_{j\geq 1}g_{ij}t^{j-1})\homo(f_i)+G\in\homo(I)$, as desired.

Now that we have proved the case in which $F$ is $t$-homogeneous, the general result follows from Lemma \ref{l:homogeneous general} and the fact that $\homo(I)$ is a $t$-homogeneous ideal.
\end{proof}

Next we want to show that $\homo(I)$ is a prime ideal of $T$ whenever $I$ is a prime ideal of $S$. The analogous fact for finitely generated $A$-algebras has a simple proof, but this does not immediately translate to our setup. The proof we present is more subtle, and we need two preparation lemmas. The fist one is well-known, the second one is trickier.

\begin{lemma}\label{l:easy general}
For any ideal $I\subseteq S$ there is an isomorphism of $A$-algebras $(S/I)\ps{t} \cong T/IT$.
\end{lemma}
\begin{proof}
If $\pi:T \to (S/I)\ps{t}$ is the natural projection $\sum_{n\in\N}f_nt^n\mapsto \sum_{n\in\N}\overline{f_n}t^n$, it is clear that $IT\subseteq \Ker(\pi)$. On the other hand, if $\sum_{n\in\N}f_nt^n\in\Ker(\pi)$, then $f_n\in I$ for all $n\in\N$. If $g_n=\sum_{i\leq n}f_it^i$ we thus have that $(g_n)_{n\in \N}$ is a Cauchy sequence in $T$ with entries in $IT$. Since any ideal of $T$ is complete, then the limit of $(g_n)_{n\in\N}$, i.e. $\sum_{n\in\N}f_nt^n$, belongs to $IT$. So $\Ker(\pi)=IT$, and it follows that $\overline{\pi}:T/IT \to (S/I)\ps{t}$ is an isomorphism.
\end{proof}

\begin{lemma}\label{l:tricky general}
Let $I \subseteq S$ be an ideal. If $\mm_S=\mm_AS + (x_1,\ldots,x_n)$ denotes the maximal ideal of $S$, then $R=S/I$ is a domain if and only if the $\mm_S R\ps{t}_t$-adic completion of $R\ps{t}_t = T_t/I T_t$ is a domain.
\end{lemma}
\begin{proof}
If the $\mm_S R\ps{t}_t$-adic completion $\widehat{R\ps{t}_t}$ is a domain, then $R$ is a domain as well since the map $R \to  \widehat{R\ps{t}_t}$ is faithfully flat, hence injective.


Suppose now that $R$ is a domain. Denote $R\ps{t}_t$ by $R'$, $\mm_S R' = \mm_S\ps{t}_t$ by $\mm'$ and the $\mm'$-adic completion by $\widehat{(-)}$ as above. Let $\overline{R}$ be the normalization of $R$. Since $R$ is a complete local domain, $\overline{R}$ is local \cite[Lemma 15.106.2]{Stacks}, say with maximal ideal $\overline{\mm}$. In particular, the ring $\overline{R}\ps{t}$ is a normal local domain, with maximal ideal $\overline{\mm}\ps{t}$. Moreover, since $R$ is complete, hence excellent, $\overline{R}\ps{t}_t$ is a finitely generated $R'$-module. Now, note the ideal $\mm'$ is maximal in $R'$ since $R'/\mm' \cong \kappa_A\ps{t}_t = \kappa_A \pps{t}$ is a field. As the map $R' \to \widehat{R'}$ is flat, if we tensor the inclusion $R' \hookrightarrow \overline{R}\ps{t}_t$ with $\widehat{R'}$ we still have an injective map
\[
\widehat{R'} \hookrightarrow \overline{R}\ps{t}_t \otimes_{R'} \widehat{R'} \cong \widehat{\overline{R}\ps{t}_t}.
\]
It is left to show that the ring on the right-hand-side is in fact the completion of $\overline{R}\ps{t}_t$ at one of its maximal ideals. As such, it will be a normal domain and, in particular, its subring $\widehat{R'}$ will be a domain as well.

In order to prove the first claim we will show that the $\mm' \overline{R}\ps{t}_t$-topology is equivalent to the $\overline{\mm}\ps{t}_t$-topology; note that the latter is a maximal ideal since $\overline{R}\ps{t}_t/\overline{\mm}\ps{t}_t \cong (\overline{R}/\overline{\mm})\pps{t}$ is a field. To this end, just recall once again that the map $R \to \overline{R}$ is finite and $\overline{R}$ is local. It follows that $\sqrt{\mm_S \overline{R}} = \overline{\mm}$, and therefore $\sqrt{\mm_S \overline{R}\ps{t}_t} = \sqrt{\mm' \overline{R}\ps{t}_t} = \overline{\mm}\ps{t}_t$.
\end{proof}

\begin{proposition} \label{p:iso}
Let $S$ and $T=S\ps{t}$ be as above, and let $S'=A'\ps{x_1,\ldots,x_n}$ with $A'=\widehat{A\pps{t}}^{\mm_A\pps{t}}$. Then $S'$ is the completion of $T_t$ at the ideal $\nn=\mm_AT_t + (x_1,\ldots,x_n)T_t$. Moreover, there is an automorphism of $A'$-algebras $\gamma:S' \to S'$ such that $\gamma(\homo(I)S') = IS'$ for any ideal $I \subseteq S$.
\end{proposition} 
\begin{proof}
By Cohen's structure theorem, we can write $A=V\ps{y_1,\ldots,y_r}$, where $V$ is either a field isomorphic to $\kappa_A$, or a complete discrete valuation domain with uniformizer $p$ and residue field isomorphic to $\kappa_A$. 
Note that, since $t \notin \nn$, the completion $\widehat{T_t}$ at the maximal ideal $\nn T_t$ is isomorphic to the completion $\widehat{T_\nn}$ of $T_\nn$ at its maximal ideal. 
Since the residue field of $T_\nn$ is isomorphic to $\kappa_A\pps{t}$, by Cohen's structure theorem, and invoking \cite[Corollary to Theorem 29.2]{MatsumuraRing} in the case of mixed characteristic, we conclude that $\widehat{T_\nn} \cong V'\ps{y_1,\ldots,y_r,x_1,\ldots,x_n}/J$ for some ideal $J$, and $V'=\widehat{V\pps{t}}^{(p)}$ the $p$-adic completion of $V\pps{t}$. Finally, by dimension considerations, we have that $J=0$ and thus $\widehat{T_\nn} \cong V'\ps{y_1,\ldots,y_r}\ps{x_1,\ldots,x_n}$. Finally, note that $A' \cong V'\ps{y_1,\ldots,y_r}$, again by Cohen's structure Theorem, and using \cite[Corollary to Theorem 29.2]{MatsumuraRing} in mixed characteristic. This shows the first claim.

Now consider the map $\gamma:S' \to S'$ of complete local $A'$-algebra defined by $x_i \mapsto x_i/t$. Since $t$ is a unit it is an automorphism, and by construction it is such that 
\[
\gamma(\homo(f)) = \gamma\left(t^{-o(f)}\psi(f)\right) = t^{-o(f)} \gamma(\psi(f)) = t^{-o(f)}f
\]
for all $0\ne f \in S$. In particular, if $f_1,\ldots,f_u$ are generators of an ideal $I$ such that $\homo(I) = (\homo(f_1),\ldots,\homo(f_u))$, then
\[
\gamma(\homo(I)S') = \left(\gamma(\homo(f_1)),\ldots,\gamma(\homo(f_u))\right)S' = \left(f_1,\ldots,f_u\right)S' = IS'. \qedhere
\]
\end{proof}

\begin{proposition}\label{p:prime general}
If $I\subseteq S=A\ps{x_1,\ldots,x_n}$ is a prime ideal, then $\homo(I)$ is a prime ideal of $T=S\ps{t}$.
\end{proposition}
\begin{proof}
Lemma \ref{l:easy general} implies that $I$ is prime if and only if $IT$ is prime. Any localization of a prime ideal at a disjoint multiplicative system is prime, hence $IT_t$ is prime. By Proposition \ref{p:iso} the completion of $T_t$ at the ideal $\nn=\mm_A T_t + (x_1,\ldots,x_n)T_t$ is isomorphic to $S'=A'\ps{x_1,\ldots ,x_n}$, where $A'=\widehat{A\pps{t}}^{\mm_A}$. Again by Proposition \ref{p:iso} we have an isomorphism $S'/IS'\cong S'/\homo(I)S'$ of $A'$-algebras, and since $IS'$ is prime by Lemma \ref{l:tricky general}, we conclude that $\homo(I)S'$ is prime as well. Finally, $\homo(I)$ is prime because the map $T/\homo(I) \to S'/\homo(I)S'$ is faithfully flat, hence injective. 
\end{proof}

\subsection{Rings of equal characteristic}
Let $(R,\mm)$ be a complete local ring containing a field, and $G=\bigoplus_{i\in\N}\mm^i/\mm^{i+1}$ be its associated graded ring. We let $k \subseteq R$ be a coefficient field of $R$. We recall the general construction introduced above, and specialize it to this case: by Cohen's structure theorem, we can write $R\cong S/I$ where $S=k\ps{x_1,\ldots ,x_n}$ is the formal powers series ring in $n = \dim_{k}(\mm/\mm^2)$ variables over the field $k\cong R/\mm$ and $I\subseteq (x_1,\ldots,x_n)^2$. We can write any element $0 \ne f\in S$ uniquely as $f=\sum_{i\in \N}f_i$ where $f_i\in P=k[x_1,\ldots,x_n]$ is a homogeneous polynomial of degree $i$. We let $\init(f):=f_{o(f)}$, where $o(f)= \min\{i \in \N \mid f \in \mm^i\}= \min\{i\in\N \mid f_i\neq 0\}$. With this notation, if $\init(I)=(\init(f) \mid f\in I)\subseteq P$ then we have $G\cong P/\init(I)$. Note that, if $G_+$ denotes the homogeneous maximal ideal of $G$, then $\widehat{G}:= \widehat{G}^{G_+}$ is naturally an $S$-module.

The ideal $\homo(I)\subseteq T=S\ps{t}$ introduced before is, in this context of $S=A\ps{x_1,\ldots,x_n}$ with $A=k$, generated by all elements of the form
\[\homo(f)=\sum_{i\geq o(f)}t^{i-o(f)}f_i\]
with $0 \ne f=\sum_{i\geq o(f)}f_i\in I$.

\begin{theorem}\label{t:deform equichar}
We have that $\homo(R):=T/\homo(I)$ is a flat $k\ps{t}$-algebra with special fiber $\homo(R)/(t)\cong \widehat{G}$ and generic fiber $\widehat{\homo(R)\otimes_{k\ps{t}} k\pps{t}}\cong \widehat{R\otimes_kk\pps{t}}$, where in both cases the completion is taken with respect to the extension of $\mm = (x_1,\ldots,x_n)$.
\end{theorem} 
\begin{proof}
Since $k\ps{t}$ is a discrete valuation ring, to prove that $\homo(R)$ is flat over $k\ps{t}$ it is enough to see that $\ann_{\homo(R)}(t)=0$. 
But this is an immediate consequence of Lemma \ref{l:nzd}. Now, to show the isomorphisms, first consider the evaluation map $\phi:T\to S$ sending $t$ to $0$. One has that $\Ker(\phi)=(t)$ and $\phi(\homo(I))= \init(I)S$. To see the latter, observe that for any $0 \ne f \in S$ one has
\[
\phi(\homo(f))=  \left(\sum_{i \geq o(f)} f_it^{i-o(f)}\right) _{|t=0} = f_{o(f)} = \init(f) \in S,
\]
and therefore the equality is clear. Thus, we have $T/(\homo(I)+(t))\cong S/\init(I)S \cong \widehat{G}$. For the last isomorphism, if we let $S'=k\pps{t}\ps{x_1,\ldots,x_n}$, then by Proposition \ref{p:iso} we have that $S'/\homo(I)S' \cong S'/IS'$ as $k\pps{t}$-algebras. Using again Proposition \ref{p:iso}, we conclude that 
\begin{align*}
\widehat{\homo(R)\otimes_{k\ps{t}} k\pps{t}}^{(x_1,\ldots,x_n)}& \cong \widehat{\frac{T_t}{\homo(I)_t}}^{(x_1,\ldots,x_n)} \cong \frac{S'}{\homo(I)S'} \\
& \cong \frac{S'}{IS'} \cong \widehat{\frac{S\pps{t}}{IS\pps{t}}}^{(x_1,\ldots,x_n)} \cong \widehat{R \otimes_k k\pps{t}}^{(x_1,\ldots,x_n)}. \qedhere
\end{align*}
\end{proof}

\subsection{Rings of mixed characteristic} 
Let $(R,\mm)$ be a complete local ring of mixed characteristic $(\alpha, p)$, with $\alpha \in \{0\} \cup \{p^n \mid n > 1\}$, and let $G$ be its associated graded ring, with completion $\widehat{G}$ at its irrelevant maximal ideal $G_+$. By Cohen's structure theorem there exists an unramified complete discrete valuation ring $(V,pV)$ of mixed characteristic $(0,p)$ and an ideal $I \subseteq  V\ps{x_2,\ldots,x_n}=:Q$ such that $R \cong Q/I$. Note that we are not excluding the case in which $R$ is ramified, as we are not assuming that $I \subseteq \mm_Q^2$. We let $S=Q\ps{x_1} \cong V\ps{x_1,\ldots,x_n}$, and we consider the surjective map of complete $Q$-algebras $\pi: S \to Q$ such that $x_1 \mapsto p$. It is clear that $(x_1 - p) = \ker(\pi)$, so that $Q \cong S/(x_1-p)$. 

Given an ideal $I \subseteq Q$, we let
\[
\whomo(I) = \homo\left(\pi^{-1}(I)\right) \subseteq T=S\ps{t},
\]
where $\hom(-)$ is the one previously defined and applied to the context of $S=A\ps{x_1,\ldots,x_n}$ with $A=V$.
\begin{remark} \label{r:p,t}
Observe that $x_1-p \in \pi^{-1}(I)$, and thus $x_1t-p \in \whomo(I)$ for any ideal $I \subseteq Q$.
\end{remark}

Now let $G$ be the associated graded ring of $R$. By definition, it is a standard graded $R/\mm=k$-algebra, which we can write as $G\cong P/\init(I)$ where $P=k[x_1,x_2,\ldots,x_n]$ and $\init(I)$ is a homogeneous ideal. It is worth pointing out how the initial ideal $\init(I)$ is obtained in this context or, in other words, how given $0 \ne f \in Q=V\ps{x_2,\ldots,x_n}$ one constructs a homogeneous polynomial $\init(f) \in P$, the initial form of $f$.

Since $V$ is an unramified complete discrete valuation domain with uniformizer $p$, and residue field $k=V/(p)$, we can choose a function $[-]: k \to V$, which we call a lift, such that if $\rho:V \to k$ is the natural projection, one has $\rho \circ [-] = 1_k$. In other words, for every $u \in k \cong V/(p)$ we make a choice of a representative $[u]\in V$ which coincides with $u$ modulo $(p)$. Note that when $k$ is perfect, there is a unique such map which is also multiplicative and every element of $\Im([-])$ is a $p^n$ power for every $n$. Such a map is often called a Teichmuller lift  (\cite{Teichmuller,Cohen}). 

We can extend any lift $[-]$ to an injective function $\iota: k\ps{x_1,\ldots,x_n} \to V\ps{x_1,\ldots,x_n} = S$ defined as 
\[
\iota \left(\sum_{\underline{a}} u_{\underline{a}} x_1^{a_1} \cdots x_n^{a_n}\right) = \sum_{\underline{a}} [u_{\underline{a}}] x_1^{a_1} \cdots x_n^{a_n}.
\]
If $\rho: S \to k\ps{x_1,\ldots,x_n}$ still denotes the natural projection, we have $\rho \circ \iota = 1_{k\ps{x_1,\ldots,x_n}}$.
Note that $\iota$ is both additive and multiplicative modulo $(p)$, that is, $\iota(f+g) - (\iota(f)+\iota(g)) \in (p)S$ and $\iota(fg)-\iota(f)\iota(g) \in (p)S$. Both of these properties follow at once from the fact that $\rho$ is a ring homomorphism with $\ker(\rho) = (p)S$, and that $\rho \circ \iota$ is the identity map.

Now given an element $f \in Q$, we can write it uniquely as $\sum_{\underline{a} \in \N^{n-1}} \lambda_{\underline{a}} x_2^{a_2} \cdots x_n^{a_n}$ for $\lambda_{\underline{a}} \in V$. Moreover, for every $\underline{a}$ we can write $\lambda_{\underline{a}} = \sum_{j \in \N} \lambda_{\underline{a},j}p^j$ for an element $(\lambda_{\underline{a},j})_{j \in \N} \in \prod_\N \Im([-])$ which is unique once a lift $[-]$ as above is fixed. Since any such lift $[-]$ is injective, we can equivalently say that we can write each $\lambda_{\underline{a}} = \sum_{j\in \N} [u_{\underline{a},j}]p^j$ for a unique element $(u_{\underline{a},j})_{j \in \N} \in \prod_\N k$, this time independent of the chosen lift $[-]$. In other words, every $f \in Q$ can be written as $f=\sum_{\underline{a}\in \N^n} [u_{\underline{a}}] p^{a_1}x_2^{a_2}\cdots x_n^{a_n}$ for a unique choice of elements $(u_{\underline{a}})_{\underline{a}\in \N^{n}} \in \prod_{\N^n}k$. We now let $v(f) = \max\{j \mid f \in (p,x_2,\ldots,x_n)^j\}$, so that we can write $f=\sum_{|\underline{a}|=v(f)} [u_{\underline{a}}] p^{a_1}x_2^{a_2}\cdots x_n^{a_n} + \sum_{|\underline{b}|>v(f)} [u_{\underline{b}}] p^{b_1}x_2^{b_2}\cdots x_n^{b_n}$. With this notation, we have that $\init(f) = \sum_{|\underline{a} |=v(f)} u_{\underline{a}} x_1^{a_1}x_2^{a_2} \cdots x_n^{a_n} \in P_{v(f)}$.

\begin{remark} Note that the writing $\init(f) = \sum_{|\underline{a} |=v(f)} u_{\underline{a}} x_1^{a_1}x_2^{a_2} \cdots x_n^{a_n} \in P_{v(f)}$ is unique, while $f=\sum_{|\underline{a}|=v(f)} [u_{\underline{a}}] p^{a_1}x_2^{a_2}\cdots x_n^{a_n} + \sum_{|\underline{b}|>v(f)} [u_{\underline{b}}] p^{b_1}x_2^{b_2}\cdots x_n^{b_n}$ is unique up to the choice of $[-]:k \hookrightarrow V$.
\end{remark}

\begin{remark}
The above writing argument can be adapted to any element $f$ of $S$, rather than just of $Q$. In other words, given any $f \in S$, there exists unique elements $v_{\underline{c}} \in k$ for $\underline{c} \in \N^{n+1}$ such that 
\[
f=\sum_{\underline{c} \in \N^{n+1}} [v_{\underline{c}}]p^{c_0}x_1^{c_1} \cdots x_n^{c_n}.
\]
\end{remark}
Given $f \in S$ written as above, we let 
\[
\bar{f} := \sum_{\underline{c} \in \N^{n+1}} [v_{\underline{c}}] x_1^{c_0+c_1}x_2^{c_2} \cdots x_n^{c_n} \in S.
\]
Note that this construction will, in general, depend on the chosen lift $[-]$. However, since $f-\bar{f} \in (p-x_1) = \ker(\pi)$, given any ideal $I \subseteq Q$, one has that $f \in \pi^{-1}(I)$ if and only if $\bar{f} \in \pi^{-1}(I)$. It follows easily from the definitions that there exists $N\geq 0$ such that $\homo(f)-t^N\homo(\bar{f}) \in (p-x_1t)$.
\begin{theorem} \label{t:deform mixed}
Set $\homo(R):=T/\whomo(I)$. Then $t$ is a non-zero divisor on $\homo(R)$, and $\homo(R)/(t)\cong \widehat{G}$. Moreover, we have $\widehat{\homo(R)\otimes_{V\ps{t}} V\pps{t}}\cong \widehat{R\otimes_V V\pps{t}}$, where on the left-hand side completion is taken with respect to the extension of $(p,x_1,\ldots,x_n)$, and on the right-hand side with respect to the extension of $\mm = (p,x_2,\ldots,x_n)$.
\end{theorem}
\begin{proof}
The fact that $t$ is a non-zero divisor on $\homo(R)$ is a direct consequence of Lemma \ref{l:nzd}. Consider the $S$-algebra map $\phi:T \to S$ such that $t \mapsto 0$. We claim that $\phi(\whomo(I)) = (\iota(\init(I)),p)S$, so that
\[
\frac{\homo(R)}{(t)} \cong \frac{T}{\whomo(I)+(t)} \cong \frac{S}{(\iota(\init(I)),p)S} \cong \frac{k\ps{x_1,\ldots,x_n}}{\init(I)k\ps{x_1,\ldots,x_n}} \cong \widehat{G},
\]
with $k\cong R/\mm$. To see the claim, recall that $\whomo(I)$ is defined as $\hom(\pi^{-1}(I))$, where $\pi:S \to Q$ is the map sending $x_1 \mapsto p$. Let $f \in S$ be such that $\pi(f) \in I$. 
Since $\homo(f)-t^N\homo(\bar{f}) \in (p-x_1t)$ for some $N \geq 0$, it suffices to show that $\phi(\homo(\bar{f})) \in (\init(I),p)S$. Thus, without loss of generality, we may assume that $f=\bar{f}$. Let $o(f) = \sup\{j \mid f \in (x_1,x_2,\ldots,x_n)^j\}$, and $v(f) = \sup\{j \mid f \in (p,x_1,x_2,\ldots,x_n)^j\}$. Note that our assumptions guarantee that $o(f) = v(f)$. We can then write 
\[
f= \sum_{|\underline{a}|=o(f)} [u_{\underline{a}}] x_1^{a_1} x_2^{a_2}\cdots x_n^{a_n} + \sum_{|\underline{b}|>o(f)} [u_{\underline{b}}] x_1^{b_1} x_2^{b_2}\cdots x_n^{b_n},
\]
so that
\[
\homo(f) = \sum_{|\underline{a}|=o(f)} [u_{\underline{a}}] x_1^{a_1} x_2^{a_2}\cdots x_n^{a_n} + \sum_{|\underline{b}|>o(f)} [u_{\underline{b}}] x_1^{b_1} x_2^{b_2}\cdots x_n^{b_n} t^{|\underline{b}|-o(f)}\]
and $\phi(\homo(f)) = \sum_{|\underline{a}|=o(f)} [u_{\underline{a}}] x_1^{a_1} x_2^{a_2}\cdots x_n^{a_n}$. On the other hand, from the fact that $o(f)=v(f)$, we deduce that $\sum_{|\underline{b}|>o(f)} [u_{\underline{b}}] x_1^{b_1} x_2^{b_2}\cdots x_n^{b_n} \in (p,x_1,x_2,\ldots,x_n)^{v(f)+1}$. We therefore have that $\pi(f) \equiv \sum_{|\underline{a}|=o(f)} [u_{\underline{a}}] p^{a_1} x_2^{a_2}\cdots x_n^{a_n} \bmod (p,x_2,\ldots,x_n)^{v(f)+1}$, and thus 
\[
\init(\pi(f)) = \sum_{|\underline{a}|=o(f)} u_{\underline{a}} x_1^{a_1} x_2^{a_2}\cdots x_n^{a_n}.
\]
This gives $\phi(\homo(f)) = \iota(\init(\pi(f))) \in (\iota(\init(I)),p)S$.

Conversely, observe that $p \in \phi(\whomo(I))$ since $p-x_1t \in \whomo(I)$. Now consider an element $f\in \iota(\init(I))$. Since $\iota$ is additive and multiplicative modulo $(p)$, it suffices to show that $f \in \phi(\overline{\homo}(I))$ for any $\iota(f)$, where $f$ is a homogeneous generator of $\init(I)$ of some degree $v$. In other words, we can assume that $\iota(f)=\sum_{|\underline{a}|=v} [u_{\underline{a}}] x_1^{a_{1}} \cdots x_n^{a_{n}} = \iota(\init(g))$ for some $g \in I$ such that $g \in \nn^v \smallsetminus \nn^{v+1}$, where $\nn=(p,x_2,\ldots,x_n)$. As we have
\[
\iota(\init(g)) = \iota\left(\sum_{|\underline{a}|=v} u_{\underline{a}} x_1^{a_1} \cdots x_n^{a_n}\right) = \sum_{|\underline{a}|=v} [u_{\underline{a}}] x_1^{a_1} \cdots x_n^{a_n}=f,
\]
this  implies that $g = \sum_{|\underline{a}|=v} [u_{\underline{a}}] p^{a_1}x_2^{a_2} \cdots x_n^{a_n} +h$ for some $h \in \nn^{v+1}$. Since $g \in I$, it follows that $\bar{g}=\sum_{|\underline{a}|=v} [u_{\underline{a}}] x_1^{a_1}x_2^{a_2} \cdots x_n^{a_n} +\bar{h} \in \pi^{-1}(I)$, and to prove the claimed equality it thus suffices to show that $\phi(\homo(\bar{g})) = f$. Note that, since $h \in \mm^{v+1}$, we must have $\bar{h} \in (x_1,x_2,\ldots,x_n)^{v+1}$. In particular, $o(\bar{h})>o(\bar{g})$, so that $\homo(\bar{g}) = \sum_{|\underline{a}|=v} [u_{\underline{a}}] x_1^{a_1} \cdots x_n^{a_n} + th'$ for some $h' \in T$. The claim now follows.

For the second isomorphism, if we let $V'=\widehat{V\pps{t}}^{(p)}$ and $S'=V'\ps{x_1,\ldots,x_n}$, then by Proposition \ref{p:iso} we have that $S'/\whomo(I)S' \cong S'/\pi^{-1}(I)S'$ as $V'$-algebras, and it follows that
\begin{align*}
\widehat{\homo(R) \otimes_{V\ps{t}} V\pps{t}}^{(p,x_1,\ldots,x_n)} & \cong \widehat{\frac{T_t}{\whomo(I)_t}}^{(p,x_1,\ldots,x_n)} \cong \frac{S'}{\whomo(I)S'} \cong \frac{S'}{\pi^{-1}(I)S'} \\
& \cong \frac{V'\ps{x_2,\ldots,x_n}}{IV'\ps{x_2,\ldots,x_n}} \cong \widehat{\frac{Q\pps{t}}{IQ\pps{t}}}^{(p,x_2,\ldots,x_n)} \cong \widehat{R \otimes_V V\pps{t}}^{(p,x_2,\ldots,x_n)}. \qedhere
\end{align*}
\end{proof}

\begin{remark} Note that, contrary to the case of equal characteristic, the extension $V\ps{t} \to \homo(R)$ is not flat. In fact, note that multiplication by $p$ is injective on $V\ps{t}/(t) \cong V$, but it becomes the zero map when we apply $-\otimes_{V\ps{t}} \homo(R)$ because of Remark \ref{r:p,t}.
\end{remark}

\begin{remark}\label{r:agree} If $R=Q/I$ with $Q=k\ps{x_2,\ldots,x_n}$ and $k$ a field of characteristic $p>0$, one could in principle repeat the construction performed in mixed characteristic, and view $R$ as a quotient of $S=V\ps{x_1,\ldots,x_n}/(x_1-p)$ with $V$ an unramified complete DVR. However, if $\pi:S \to Q$ denotes the map $x_1 \mapsto p$, then since $0=p \in I$ one gets that $x_1 \in \pi^{-1}(I)$, and thus $(p,x_1) \subseteq \whomo(I)$. It follows that $T/\whomo(I) \cong Q\ps{t}/\homo(I)$ gives the same construction as the one described in equal characteristic.
\end{remark}

\section{Subdimensions and connectedness}\label{sconnsub}
Throughout this section, $A$ denotes a Noetherian ring, not necessarily local.

\begin{definition}
The {\it subdimension} of $A$ is defined as $\gsdim(A)=\min\{\dim(A/\pp) \mid \pp\in\Min(A)\}$. Similarly, the {\it algebraic subdimension} of $A$ is $\asdim(A)=\min\{\dim(A/\pp) \mid \pp\in\Ass_A(A)\}$. We say that $A$ is {\it equidimensional} if $\gsdim(A) = \dim(A)$.
\end{definition}

\begin{remark}
By definition, we have $\asdim(A)\leq \gsdim(A)$. Moreover, equality holds if $A$ has no embedded associated primes (e.g., if $A$ is reduced). On the other hand, note that $\dim(A)=\max\{\dim(A/\pp) \mid \pp\in\Min(A)\}=\max\{\dim(A/\pp) \mid \pp\in\Ass_A(A)\}$.
\end{remark}

\begin{definition} \label{d:connectedness}
Let $A$ be a Noetherian ring of finite Krull dimension $d$. For every $s\in\N$, we say that $A$ is {\it connected in codimension $s$} if $\Spec(A)\smallsetminus Z$ is connected for all closed subsets $Z \subseteq \Spec(A)$ with $\dim(Z)<d-s$ (we use the convention that $\emptyset$ is disconnected of dimension $-\infty$).
\end{definition}



\begin{remark} \label{remark codim VS dim}
If $Z$ is a closed subset of $\Spec(A)$, the condition $\dim(Z)<\dim(A)-s$ is sometimes replaced with $\codim(Z) > s$, c.f. \cite{HartCICo}. The two definitions coincide if $\height(I)+\dim(A/I)=\dim(A)$ for any ideal $I\subseteq A$, e.g., if $A$ is catenary and equidimensional. Our definition adapts better to use the following connectedness result due to Grothendieck \cite[Expos\'e XIII, Th\'eor\`eme 2.1]{SGA2}.
\end{remark}

\begin{theorem}\label{t:grosdim}
Let $(A,\mm)$ be a Noetherian complete local ring, and $x\in\mm$ an element which is not in any minimal prime of $A$. Then 
\begin{enumerate}
\item $\gsdim(A/(x))\geq \gsdim(A)-1$.
\item If $A$ is connected in codimension $s>0$, then so is $A/(x)$.
\end{enumerate}
\end{theorem}

We now give an equivalent interpretation of being connected in a given codimension $s \in \N$. Let $\Gamma_s(A)$ denote the (finite) simple graph with vertices the minimal prime ideals of $A$, and edges $\{\pp,\qq\}$ whenever $\dim\left(A/(\pp+\qq) \right)\geq d-s$ (we use the convention that the zero ring has dimension $-\infty$).

A proof of the following proposition can be found in \cite[Proposition 2.2.4]{BCRV}.

\begin{proposition}\label{p:basicconn}
Let $A$ be a $d$-dimensional Noetherian ring. For any $s\in\N$ the following are equivalent:
\begin{itemize}
\item[(i)] $A$ is connected in codimension $s$.
\item[(ii)] The graph $\Gamma_s$ is connected. 
\item[(iii)] $\dim\left(\frac{A}{\bigcap_{\pp\in U}\pp+\bigcap_{\pp\in V}\pp}\right)\geq d-s$ for any partition $(U,V)$ of $\Min(A)$ with $U\neq \emptyset\neq V$.
\end{itemize}
\end{proposition}

We collect some basic facts about the graphs $\Gamma_s$.
\begin{remark}
The graphs $\Gamma_s=\Gamma_s(A)$ form a sequence of graphs on the vertex set $\Min(A)$ such that:
\begin{itemize}
\item[(i)] $\Gamma_0\subseteq \Gamma_1\subseteq \ldots\subseteq \Gamma_d\subseteq \Gamma_{d+1}=\Gamma_{d+2}=\ldots$ (in particular, if $A$ is connected in codimension $s$ then it is connected in codimension $s'$ for all $s' \geq s$).
\item[(ii)] $\Gamma_0$ has only isolated vertices. In particular, $A$ is connected in codimension $0$ if and only if $A/\sqrt{(0)}$ is a domain.
\item[(iii)] $\Gamma_{d+1}$ is the complete graph.
\item[(iv)] If $A$ is local, then $\Gamma_d$ is the complete graph.
\end{itemize}
The graph $\Gamma_1$ is sometimes called the  {\it Hochster-Huneke graph} of $A$ or the {\it dual graph} of $\Spec(A)$.
Besides $\Gamma_1$, the graph $\Gamma_{d-1}$ is of particular interest when $(A,\mm)$ is local, as it detects the connectedness of the {\it punctured spectrum} $\Spec(A)\setminus \{\mm\}$ (see Proposition \ref{p:basicconn}).
\end{remark}

\begin{theorem}\label{t:groconn}
Let $(A,\mm)$ be a Noetherian complete local ring, $x\in\mm$ be an element that is not in any minimal prime of $A$, and $s>0$. If $A$ is connected in codimension $s$, then $A/(x)$ is connected in codimension $s$ as well. If $(x)$ is radical, the converse holds true.
\end{theorem}

\begin{proof}
All statements are not affected by passing from $A$ to $A/\sqrt{0}$, therefore we may assume that $x$ is regular on $A$. The first part of the statement is an immediate consequence of Grothendieck's connectedness' theorem \cite[Expos\'e XIII, Th\'eor\`eme 2.1]{SGA2}. Concerning the second part, the same argument used to prove \cite[Theorem 2.11]{ALNBRM} works: notice that we do not need the equidimensionality assumption on $A$ because our definition of the graphs $\Gamma_s$ differs from the one of \cite{ALNBRM}. We repeat the argument in our setting for convenience of the reader.

By Proposition \ref{p:basicconn} we need to show that, given a partition $(U,V)$ of $\Min(A)$ with $U\neq \emptyset\neq V$ we have $\dim\left(A/(I+J)\right)\geq d-s$, where $I=\bigcap_{\pp\in U}\pp$ and $J=\bigcap_{\pp\in V}\pp$. To this end, we notice that
\begin{equation}\label{eqrad}
I\cap J+(x)=(I+(x))\cap (J+(x)).
\end{equation}
Indeed, the above equality holds up to radical, because if $f=a_0+a_1x=b_0+b_1x$ with $a_0\in I$, $b_0\in J$, $a_1,b_1\in A$, then $f^2=a_0b_0+(a_0b_1+a_1b_0+a_1b_1x)x\in I\cap J+(x)$. Moreover we have a chain of inclusions:
\[(x)\subseteq I\cap J+(x)\subseteq (I+(x))\cap (J+(x))\subseteq \sqrt{I\cap J+(x)}= \sqrt{\sqrt{(0)}+(x)}=\sqrt{(x)}.\]
Since $(x)=\sqrt{(x)}$, the inclusion $I\cap J+(x)\subseteq (I+(x))\cap (J+(x))$ has to be an equality.

Now we claim that $x\notin \pp$ for any $\pp\in\Ass_A(A/(I+J))$. Notice that, since $x\notin \pp$ for any $\pp\in\Ass_A(A)$ and $\Ass_A(A/I)\cup\Ass_A(A/J)\subseteq \Ass_A(A)$, we have that $x\notin \pp$ for any $\pp\in\Ass_A(A/I)\cup\Ass_A(A/J)$. This is equivalent to saying that $\Tor_1^A(A/I,A/(x))=\Tor_1^A(A/J,A/(x))=0$. But then, tensoring the short exact sequence
\[0\to A/I\cap J\to A/I\oplus A/J\to A/(I+J)\to 0\]
with $A/(x)$ gives rise to the exact sequence
\[0\to \Tor_1^A(A/(I+J),A/(x))\to A/(I\cap J+(x))\to A/(I+(x))\oplus A/(J+(x))\to A/(I+J+(x))\to 0.\]
By equality \eqref{eqrad} we infer that $\Tor_1^A(A/(I+J),A/(x))=0$,  and thus  $x\notin \pp$ for all $\pp\in\Ass_A(A/(I+J))$.

Denoting by $M(\pp)=\Min(\pp+(x))$, we have $\Min(R/(x))=\cup_{\pp\in\Min(A)}\{\qq/(x) \mid \qq\in M(\pp)\}$. Exploiting what we proved so far we can conclude because
\begin{align*}
\dim \left(\frac{A}{I+J} \right)&= \dim \left(\frac{A}{I+J+(x)}\right)+1\geq\dim \left(\frac{A}{\bigcap_{\pp\in U}(\pp+(x))+\bigcap_{\pp\in V}(\pp+(x))}\right)+1\\
&\geq \dim \left(\frac{A}{\bigcap_{\pp\in U}\bigcap_{\qq\in M(\pp)}\qq+\bigcap_{\pp\in V}\bigcap_{\qq\in M(\pp)}\qq}\right)+1\\
& =\dim \left(\frac{A/(x)}{\bigcap_{\pp\in U}\bigcap_{\qq\in M(\pp)}\qq/(x)+\bigcap_{\pp\in V}\bigcap_{\qq\in M(\pp)}\qq/(x)}\right)+1\\
&\geq (d-1)-s+1=d-s.
\end{align*}
where the last inequality follows by Proposition \ref{p:basicconn} because $A/(x)$ is connected in codimension $s$. 
\end{proof}

\begin{remark}
When $s=0$, the first part of Theorem \ref{t:groconn} is easily seen to be false. In fact, a ring $A$ is connected in codimension $0$ if and only if $A/\sqrt{(0)}$ is a domain. The second part of the statement remains true, because the condition of being a domain deforms. Notice that, without assuming that $(x)$ is a radical ideal, even if $A/(x)$ is connected in codimension $0$ it may be not true that $A$ is connected in codimension $0$: For example, let $A=k\ps{a,b}/(ab)$ where $k$ is a field, and $x=\overline{a-b} \in A$. 
\end{remark}

An immediate consequence of Theorem \ref{t:groconn} is the following:

\begin{corollary} \label{c:PSpec}
Let $(A,\mm)$ be a Noetherian complete local ring, and assume that there exists an element of $\mm$ that is not in any minimal prime of $A$, and that generates a radical ideal. Then the punctured spectrum $\Spec(A)\setminus \{\mm\}$ is connected.
\end{corollary}
\begin{proof}
We may harmlessly replace $A$ with $A/\sqrt{0}$ and assume that $x$ is regular on $A$. Set $d=\dim(A)$. 
Since $A/(x)$ is a local ring of dimension $d-1$, it is trivially connected in codimension $d-1$. By Theorem \ref{t:groconn}, since $(x)$ is radical we have that $A$ is also connected in codimension $d-1$, and by Proposition \ref{p:basicconn} this is equivalent to the punctured spectrum of $A$ being connected.
\end{proof}

We note that Corollary \ref{c:PSpec} follows also from \cite[Proposition 2.1]{HartCICo} since the assumption that a parameter $x$ generates a radical ideal implies that $\depth(A/\sqrt{0}) \geq 2$.

\begin{remark}
In \cite[Theorem 2.11]{ALNBRM} there is an extra assumption that $A$ is equidimensional, while the converse direction of Theorem \ref{t:groconn} does not require it. This is because in \cite{ALNBRM} the graphs $\Gamma_s$ are defined by means of heights of sums of minimal prime ideals of $A$ rather than in terms of dimensions of quotient rings defined by them (see Remark \ref{remark codim VS dim}).
\end{remark}

\begin{theorem}
Let $(A,\mm)$ be a Noetherian complete local ring, and $x\in\mm$ be an element which is not in any minimal prime of $A$. If $(x)$ is radical, then for $s>0$ we have that $\#\Gamma_s(A) = \#\Gamma_s(A/(x))$, where $\#$ denotes the number of connected components.
\end{theorem}
\begin{proof}
Let $d=\dim(A)$, and $\Sigma_1,\ldots,\Sigma_t$ be the connected components of $\Gamma_s(A)$. If we consider 
\[
\Sigma'_i = \left\{\qq/(x) \mid \qq \in \bigcup_{\pp \in \Sigma_i} M(\pp)\right\},
\]
then for $i=1,\ldots,t$ each $\Sigma_i'$ is a connected subset of $\Gamma_s(A/(x))$. In fact, let $I = \bigcap_{\pp \in \Sigma_i} \pp$, so that $\Sigma_i \cong \Gamma_s(A/I)$, which is connected by assumption. Since $\Min(I+(x)) = \bigcup_{\pp \in \Min(I)} M(\pp)$, we have that
\[
\Gamma_s\left(\frac{A}{I+(x)}\right) = \Gamma_s\left(\frac{A}{\bigcap_{\pp \in \Min(I)} \bigcap_{\qq \in M(\pp)}\qq} \right) = \Gamma_s\left(\frac{A/(x)}{\bigcap_{\pp \in \Min(I)} \bigcap_{\qq \in M(\pp)}\qq/(x)}\right) \cong \Sigma_i'.
\]
By Theorem \ref{t:groconn} we have that $\Gamma_s(A/(I+(x)))$ is connected, and so is $\Sigma_i'$. Finally, in order to show that there are no edges connecting $\Sigma_i'$ and $\Sigma_j'$ for $i \ne j$, let $(U',V')$ the partition of $\Min(A/(x))$ where $\emptyset \ne U'$ contains the minimal primes in $\Sigma_i'$, and $V' \ne \emptyset$ all the others. Set $I'=\bigcap_{\pp' \in U'} \pp'$ and $J'=\bigcap_{\pp' \in V'} \pp'$. If there was an edge between an element of $\Sigma_i'$ and an element of $\Sigma_j'$ for $j \ne i$, then we would have $\dim(A/(x)/(I'+J')) < (d-1)-s$. If we let $U= \bigcup_{\pp \in \Sigma_i} \pp$ and $V=\Min(A) \smallsetminus U$, then by construction we have that $U'= \{\qq/(x) \mid \qq \in M(\pp), \pp \in U\}$ and $V' = \{\qq/(x) \mid \qq \in M(\pp), \pp \in V\}$. Moreover, if we let $I = \bigcap_{\pp \in U} \qq$ and $J = \bigcap_{\qq \in V} \qq$, then $\sqrt{I+(x)} = \bigcap_{\pp \in U} \bigcap_{\qq \in M(\pp)} \qq$ and $I' = \bigcap_{\pp \in U} \bigcap_{\qq \in M(\pp)} \qq/(x) = \sqrt{I+(x)}/(x)$. Similarly for $J$. We therefore have that
\begin{align*}
\dim\left(\frac{A}{I+J}\right) & = \dim\left(\frac{A}{I+J+(x)}\right) + 1 \\
& = \dim \left(\frac{A}{\sqrt{I+(x)} + \sqrt{J+(x)}}\right) + 1 \\
& = \dim\left(\frac{A}{\bigcap_{\pp \in U} \bigcap_{\qq \in M(\pp)} \qq + \bigcap_{\pp \in V} \bigcap_{\qq \in M(\pp)} \qq}\right) +1\\
& = \dim\left(\frac{A/(x)}{\bigcap_{\pp \in U} \bigcap_{\qq \in M(\pp)} \qq/(x) + \bigcap_{\pp \in V} \bigcap_{\qq \in M(\pp)} \qq/(x)}\right) +1 \\
& = \dim\left(\frac{A/(x)}{I'+J'}\right) + 1.
\end{align*}
Since $\Sigma_i \cong \Gamma_s(A/I)$ is connected, we have that $\dim(A/(I+J)) \geq d-s$. It follows that $\dim(A/(x)/(I'+J')) \geq (d-1)-s$, and thus there are no edges between vertices in $U'$ and vertices in $V'$. It follows that $\Sigma_1',\ldots,\Sigma_t'$ are the connected components of $\Gamma_s(A/(x))$, and the proof is complete.
\end{proof}

\section{Connectedness of the associated graded ring} \label{Section main}
Let $(A,\mm_A,\kappa_A)$ be a complete unramified regular local ring, and $S=A\ps{x_1,\ldots,x_n}$. Let $I \subseteq S$ be an ideal, and $\homo(I) \subseteq T=S\ps{t}$ be the homogenization defined in Section \ref{Section deformation}. It is immediate to see that if $F \in T$ is $t$-homogeneous of degree $d_1$ and $G \in T$ is $t$-homogeneous of degree $d_2$, then $FG$ is $t$-homogeneous of degree $d_1+d_2$. Since every element $F\in T$ can be written uniquely as $F=\sum_{d\in \Z} F_d$ for some $F_d \in T$ which is $t$-homogeneous of degree $d$, the following is a subring of $T$:
\[
\overline{T} = \left\{ \sum_{d \geq N} F_d \ \bigg| \ N \in \Z, F_d \in T \text{ is } t\text{-homogeneous of degree } d\right\}.
\]
Consider the map $\Theta: \overline{T} \to S$ sending $t \mapsto 1$ and $x_i \mapsto x_i$ for all $i =1,\ldots,n$. Then $\Theta$ is a surjective ring homomorphism (note that it is not defined on $T$).

\begin{lemma}\label{l:dehom general}
Let $\Theta:\overline T \to S$ be as above, and $I \subseteq S$ be an ideal. Then $\Theta(\homo(I)\cap \overline T)=I$.
\end{lemma}
\begin{proof} Recall that, for any $f \in S$, the homogenization $\homo(f) \in T$ is $t$-homogeneous of degree $o(f)$. Moreover, since $\Theta(\homo(f))=\Theta(t^{-o(f)}\psi(f)) = f$ for any $f\in S$, with $\psi$ the map defined in Section \ref{Section deformation}, the inclusion $I\subseteq \Theta(\homo(I)\cap \overline T)$ is clear. For the other inclusion, let $F=\sum_{d\geq N}F_d\in \homo(I)\cap \overline T$, so that we can write $F=\sum_{i=1}^rG_i\homo(f_i)$ for some $f_i\in I$ and $G_i\in T$. If we write $G_i=\sum_{d\in\ZZ}G_{i,d}$, where each $G_{i,d}\in T$ is $t$-homogeneous of degree $d$, then after setting $\overline{G_i}=\sum_{d\geq N}G_{i,d} \in \overline{T}$ we still have that $F=\sum_{i=1}^r\overline{G_i}\homo(f_i)$. It follows that $\Theta(F)=\sum_{i=1}^r\Theta(\overline{G_i})f_i\in I$, as desired.
\end{proof}

\begin{remark}\label{dimext general}
Let $A'=\widehat{A\pps{t}}^{\mm_A\pps{t}}$, and set $S=A\ps{x_1,\ldots ,x_n}$ and $S'=A'\ps{x_1,\ldots ,x_n}$. For any ideal $I \subseteq S$ we have that $\dim(S/I)=\dim(S'/IS')$. Indeed, the natural map $A \to A'$ is faithfully flat or relative dimension $0$, and therefore so are the induced maps $S\to S'$ and $S/I \to S'/I S'$. 
\end{remark}

\begin{proposition}\label{homoprop general}
Let $I,J$ be two ideals of $S=A\ps{x_1,\ldots,x_n}$, and let $T=S\ps{t}$. We have:
\begin{enumerate}
\item $\homo(I\cap J)=\homo(I)\cap \homo(J)$.
\item $I\subseteq J$ if and only if $\homo(I)\subseteq \homo(J)$.
\item $\height_S(I)+\height_S(J)=\height_T(\homo(I))+\height_T(\homo(J))$.
\item $I$ is a prime ideal of $T$ if and only if $\homo(I)$ is a prime ideal of $T$.
\item $\homo(\sqrt{I})=\sqrt{\homo(I)}$.
\item There is a $1$-$1$ correspondence between the minimal prime ideals of $I$ and those of $\homo(I)$.
\item $\Gamma_s(S/I)=\Gamma_s(T/\homo(I))$ for any $s \in \N$.
\end{enumerate}
\end{proposition}
\begin{proof}
Let $\overline{T} \subseteq T$ be the subring defined above, together with the ring epimorphism $\Theta: \overline T\to S$. 

To prove $(1)$, note that the containment $\homo(I\cap J) \subseteq \homo(I) \cap \homo(J)$ is trivial. For the other inclusion take $F\in\homo(I)\cap\homo(J)$. Using Lemma \ref{l:homogeneous general} it is harmless to assume that $F$ is $t$-homogeneous of degree $d$. So $F\in \overline{T}$, and therefore $f:=\Theta(F)\in I\cap J$ by Lemma \ref{l:dehom general}. Note that we must have $d\leq o(f)$, and it follows that $F=t^{o(f)-d}\homo(f)\in\homo(I\cap J)$.

For $(2)$, the forward implication is clear. For the converse, let $f\in I$. Then $\homo(f)\in \homo(I)\cap \overline T\subseteq \homo(J)\cap \overline T$. So $f=\Theta(\homo(f))\in J$ by Lemma \ref{l:dehom general}.

Now we prove $(3)$. Let $A'=\widehat{A\pps{t}}^{\mm_A\pps{t}}$, and $S'=A'\ps{x_1,\ldots,x_n}$. By Remark \ref{dimext general} we have that $\height_S(I+J)=\height_{S'}((I+J)S')$. By Proposition \ref{p:iso} there is an automorphism $\gamma:S'\to S'$  of $A'$-algebras such that $\gamma(\homo(I)S')=IS'$ and $\gamma(\homo(J)S')=JS'$. Hence $\gamma(\homo(I)S'+\homo(J)S')=IS'+JS'$, and therefore $\height_{S'}((\homo(I)+\homo(J))S')=\height_{S'}((I+J)S')$. Being $S'$ the completion of $T_\nn$ at its maximal ideal, with $\nn=\mm_AT+(x_1,\ldots,x_n)T$, we have that $\height_{T_\nn}((\homo(I)+\homo(J))T_\nn)=\height_{S'}((\homo(I)+\homo(J))S')$. We conclude since $\height_{T_\nn}((\homo(I)+\homo(J))T_\nn)=\height_T(\homo(I)+\homo(J))$. 

Now to $(4)$, if $I \subseteq S$ is prime, then $\homo(I) \subseteq T$ is prime by Proposition \ref{p:prime general}. On the other hand, if $\homo(I)$ is prime in $T$, then $\homo(I)\cap \overline T$ is prime in $ \overline T$. Suppose $f,g\in S$ are such that $fg\in I$. Then $\homo(fg)=\homo(f)\homo(g)\in\homo(I)\cap \overline T$, and $\homo(f),\homo(g)\in \overline T$. It follows that either $\homo(f)$ or $\homo(g)$ are in $\homo(I)\cap \overline T$. Suppose $\homo(f) \in \homo(I)\cap \overline T$. Then $f=\Theta(\homo(f))\in I$ by Lemma \ref{l:dehom general}.

To prove $(5)$, write $\sqrt{I}=\bigcap_{\pp\in\Min(I)}\pp$. Then using (1) we have $\homo(\sqrt{I})=\bigcap_{\pp\in\Min(I)}\homo(\pp)$. Using $(4)$, we get that $\homo(\sqrt{I})$ is a radical ideal containing $\homo(I)$, so $\homo(\sqrt{I})\supseteq \sqrt{\homo(I)}$. On the other hand, let $F$ be a $t$-homogeneous element of degree $d$ of $\homo(\sqrt{I})$. By Lemma \ref{l:dehom general} we have that $f:=\Theta(F)\in \sqrt{I}$. By assumption there exists $N\in\N$ such that $f^N=\Theta(F^N)\in I$. Since $d\leq o(f)$, as already noted, we conclude that $F^N=t^{N(o(f)-d)}\homo(f^N)\in\homo(I)$, so that $F\in \sqrt{\homo(I)}$.

In order to show $(6)$, let $\pp_1,\ldots ,\pp_m$ be the minimal primes of $I$. By (5) and its proof we get that $\sqrt{\homo(I)}=\bigcap_{i=1}^s\homo(\pp_i)$. Furthermore, $\homo(\pp_i)$ is a prime ideal of $T$ for all $i=1,\ldots ,s$ by (4), and $\homo(\pp_i)\not\subseteq \homo(\pp_j)$ for any $i\neq j$ by (2). Hence the minimal primes of $\homo(I)$ are $\homo(\pp_1),\ldots ,\homo(\pp_m)$.

Finally, $(7)$ follows from $(6)$, $(3)$, and the definition of the graph $\Gamma_s$ after observing that $\dim(S/J)=\dim(S)-\height_S(J)$ for any ideal $I\subseteq J\subseteq S$, that $\dim(T/L)=\dim(S)+1-\height_T(L)$ for any ideal $\homo(I)\subseteq L\subseteq T$ and that $\dim(S/I)=\dim(T/\homo(I))-1$. 
\end{proof}

\begin{remark} \label{r:complete} Let $(R,\mm,k)$ be a complete local ring. By Cohen's structure theorem, we can write $R=Q/I$, where either $Q=k\ps{x_1,\ldots ,x_n}$ or $Q=V\ps{x_2,\ldots,x_n}$, with $V$ an unramified complete DVR. In the first case, let $S=Q$ and $\homo(I) \subseteq T=S\ps{t}$; in the second, let $S=Q\ps{x_1} = V\ps{x_1,\ldots,x_n}$, $T=S\ps{t}$, and $\whomo(I) \subseteq T$ as defined in Section \ref{Section deformation}. Recall that $\whomo(I) = \homo(\pi^{-1}(I))$, so we can still apply the properties of $\homo(-)$ we have just proved also to $\whomo(-)$. In particular, one can check that all the items of Proposition \ref{homoprop general} can be adapted verbatim to $\whomo(-)$, with the exception of (3) where a correction factor is needed. It can be checked, however, that this correction does not affect the considerations that follow in this article. Finally, we let $\homo(R)=T/\homo(I)$ in the first case, and $\homo(R)=T/\whomo(I)$ in the second. 
\end{remark}

We recall the following, see \cite[Lemma 2.1.1]{BCRV}:

\begin{lemma}\label{l:algsdim general}
Let $A$ be a Noetherian ring and $\aaa$ an ideal of $A$, for any $h\in\N$, the ideal $\aaa_{>h}=\{x\in A \mid \height(\aaa:_Ax)>h\}$ is equal to the intersection of the primary components of $\aaa$ of height at most $h$.
\end{lemma}

\begin{proposition} \label{p:algsdim general}
Let $(R,\mm,k)$ be a local ring which is the homomorphic image of a Gorenstein local ring, and $G = {\rm gr}_\mm(R)$ be its associated graded ring. Then $\asdim(G)\leq \asdim(R)$.
\end{proposition}
\begin{proof}
Notice that $\asdim(\widehat{R})=\asdim(R)$ using \cite[Lemma 2.3.10]{BCRV}. Therefore we can assume that $R$ is complete, and thus $R=Q/I$ where either $Q=k\ps{x_1,\ldots ,x_n}$, or $Q=V\ps{x_2,\ldots,x_n}$ for some unramified complete DVR $V$ with uniformizer $p$. 
Either way, $G\cong P/\init(I)$ where $P=k[x_1,\ldots ,x_n]$ and $\init(I) = (\init(f) \mid 0 \ne f \in I)$ is the initial ideal of $I$ as recalled earlier in Section \ref{Section deformation}. Notice that, for any $0 \ne f\in Q$, $\init(I:_Qf)\subseteq \init(I):_P\init(f)$. In particular,
\[
\height(I:_Qf)=\height(\init(I:_Qf))\leq \height(\init(I):_P\init(f)).
\]
It follows that $\init(I_{>h})\subseteq (\init(I))_{>h}$ for all $h\in\N$. If $e=\asdim(G)$, using Lemma \ref{l:algsdim general} we deduce that
\[\init(I)\subseteq \init(I_{>n-e})\subseteq \init(I)_{>n-e}=\init(I),\]
and therefore $I=I_{>n-e}$. Using Lemma \ref{l:algsdim general} once again, we obtain that $\asdim(R)\geq e$.
\end{proof}

We recall that a ring $A$ is catenary if, given two prime ideals $\pp \subseteq \qq$, every strictly increasing chain of primes between them can be refined to a saturated finite chain, and all saturated chains between $\pp$ and $\qq$ have the same length. Moreover, $A$ is said to be universally catenary if every finitely generated $A$-algebra is catenary. Examples of universally catenary rings include complete local rings, or Cohen-Macaulay rings; furthermore, any localization of a Noetherian universally catenary ring is universally catenary (e.g., see \cite[Lemma 10.105.4]{Stacks}). 

Our next goal is to relate the subdimensions of $R$ and of $G$ using the results on Gr{\"o}bner deformations established in Section \ref{Section deformation}, together with Theorem \ref{t:grosdim}.

\begin{theorem}\label{t:algsdim}
Let $(R,\mm)$ be a universally catenary local ring, and $G$ be its associated graded ring. Then $\gsdim(G)\geq \gsdim(R)$. Moreover, $\gsdim(G)=\gsdim(R)$ if $G$ satisfies Serre's condition $(S_1)$.
\end{theorem}
\begin{proof}
Our assumptions on $R$ guarantees that $\gsdim(R) = \gsdim(\widehat{R})$ \cite[Lemma 19.3.1. (iv)]{BroSharp}. Therefore, we may assume that $R$ is complete. Let $\homo(R)$ be defined as in Remark \ref{r:complete}. By Proposition \ref{homoprop general} (3) and (6) we have that $\gsdim(\homo(R))=\gsdim(R)+1$. Since $\homo(R)$ is a complete local ring, by Theorem \ref{t:grosdim} we deduce that $\gsdim (\homo(R)/(t))\geq \gsdim (\homo(R))-1=\gsdim(R)$. Since $\widehat{G}\cong \homo(R)/(t)$ by Theorems \ref{t:deform equichar} and \ref{t:deform mixed}, we infer the first inequality exploiting \cite[Lemma 19.3.1 (iii)]{BroSharp}.

If $G$ satisfies $(S_1)$, then using Proposition \ref{p:algsdim general} and the above inequality we conclude that
\[
\gsdim(G)=\asdim(G)\leq \asdim(R)\leq \gsdim(R)\leq \gsdim(G),
\]
forcing equalities everywhere.
\end{proof}

\begin{corollary} \label{c:equidimensional}
Let $(R,\mm)$ be a universally catenary local ring, and $G$ be its associated graded ring. If $R$ is equidimensional, then $G$ is equidimensional. If $G$ is reduced, then the converse holds true as well.
\end{corollary}

\begin{remark} The fact that if $R$ is equidimensional then so is $G$ was already known thanks to work of Ratliff \cite[Remark (A.11.4) $\iff$ (A.11.26)]{Ratliff}.
\end{remark}

We now state the main result of this section, relating connectedness properties of $R$ and $G$. We recall that a local ring is called analytically irreducible if its completion is a domain. 

\begin{theorem}\label{t:connectedness}
Let $(R,\mm)$ be a local ring, and $G$ be its associated graded ring. Assume that $R/\pp$ is analytically irreducible for every $\pp \in \Min(R)$. Let $s>0$ be an integer. If $R$ is connected in codimension $s$, then $G$ is connected in codimension $s$. If $G$ is reduced, then the converse holds true as well.
\end{theorem}
\begin{proof}
By \cite[Lemma 19.3.1 (ii)]{BroSharp} we may assume that $R$ is complete. Let $\homo(R)$ be defined as in Remark \ref{r:complete}. By Proposition \ref{homoprop general} $(7)$ we have that $R$ is connected in codimension $s$ if and only if $\homo(R)$ is connected in codimension $s$. Since $\homo(R)$ is a complete local ring and $t$ is regular on $\hom(R)$, by Theorem \ref{t:groconn} we have that $\homo(R)/(t)$ is connected in codimension $s$. By Theorems \ref{t:deform equichar} and \ref{t:deform mixed} we have that $\homo(R)/(t) \cong \widehat{G}$, and the first inequality follows then from \cite[Lemma 19.3.1 (i)]{BroSharp}.

Now note that $G$ is connected in codimension $s$ if and only if $\widehat{G}$ is connected in codimension $s$ by \cite[Corollary 2.2.10]{BCRV}. If $G$ is reduced, then $(0)=\bigcap_{\pp\in\Min(G)}\pp$. Note that by \cite[Lemma 2.2.9]{BCRV} $\pp\widehat{G}$ is a prime ideal of $\widehat{G}$ for any $\pp\in\Min(G)$, and since the extension $G\to \widehat{G}$ is flat, we have that $(0)=\bigcap_{\pp\in\Min(G)}\pp \widehat{G}$. In particular, $\widehat{G}\cong \homo(R)/(t)$ is reduced and connected in codimension $s$. By the second part of Theorem \ref{t:groconn} we have that $\homo(R)$ is connected in codimension $s$, and we conclude by Proposition \ref{homoprop general} $(7)$ that $R$ is connected in codimension $s$ as well.
\end{proof}

\begin{remark} We point out that the assumptions on $R$ that we made in \ref{t:algsdim}, \ref{c:equidimensional} and \ref{t:connectedness}, namely, that $R$ is either universally catenary or $R/\pp$ is analytically irreducible for all $\pp \in \Min(R)$, are trivially satisfied if $R$ is already assumed to be complete. Moreover, if $R$ is the localization of a non-negatively graded $R_0=k$-algebra at the homogeneous maximal ideal, then it is universally catenary and it satisfies the assumptions of Theorem \ref{t:connectedness} \cite[Lemma 2.2.9]{BCRV}.
\end{remark}

We end the section showing that the assumption that $G$ is reduced in the converse statement of Theorem \ref{t:connectedness} cannot be removed.

\begin{example} Let $k$ be a field, and $R=k\ps{x,y,z}/(x(x+y^2),xz)$. We have that $G \cong k[x,y,z]/(x^2,xz)$ is connected in codimension $0$, hence connected in codimension $1$, but $R$ is not connected in codimension $1$. Note that $G$ is not reduced.
\end{example}

\section{Bounds on numerical invariants of local domains} \label{Section bounds}

Let $(R,\mm)$ be a local ring of dimension $d=\dim(R) \geq 2$ and   $G={\rm gr}_\mm(R)$ denotes the associated graded ring. We recall that the Hilbert function of $R$ is defined as the Hilbert function of $G, $ that is $\HF_R(j) = \HF_G(j)= \dim_k\left(\mm^j/ \mm^{j+1}\right)$. Then the Hilbert series $\HS_G(t)= \sum_{j\ge0} \HF_G(j) t^j=  \frac{h(t)}{(1-t)^d}$ where $h(t) = \sum_{i \geq 0}h_it^i \in \ZZ[t]$ and $d=\dim(G)= \dim(R)$. For $i \geq 0$ we set $\e_i(R):=\e_i(G) = \frac{d^i(h(t))}{dt^i}|_{t=1}$ to be the $i$-th Hilbert coefficient of the graded algebra $G$. In particular, $\e_0(R) = \e(R) = h(1)$ is the multiplicity. 
We recall that, if $J$ is an $\mm$-primary ideal, the Hilbert-Samuel function $i \mapsto \ell_R(R/J^{j+1})$ of $R$ coincides for $j \gg 0$ with a polynomial of degree $d$, namely its Hilbert polynomial:
\[
\HP_{R,J}(j) = e_0(J)\binom{j+d}{d} -\e_1(J)\binom{j+d-1}{d-1} + \ldots +(-1)^d\e_d(J).
\] 
The integers $\e_i(J)$ are called Hilbert coefficients of $J$. When $J = \mm$ we write $\HP_{R}$ for $\HP_{R,\mm}$. Observe that $\e_i(\mm)$ coincides with $\e_i(R)$ as defined above; in particular, $\e_0(\mm) = \e(R)$ is the multiplicity. Finally, we let $\delta(R) = \HF_R(1)-d$ denote the embedding codimension of $R$, that is, the difference between the minimal number of generators of $\mm$ and its height.

If $R$ is Cohen-Macaulay, several inequalities relating the above invariants are classically known to hold:
\begin{itemize}
\item (A) $\e(R) \geq \delta(R)+1$  \cite{Abhyankar},
\item (N) $\e_1(R) \geq \e(R)-1$ \cite{Northcott},
\item (S) $\e_2(R) \geq \e_1(R)-\e(R)+1 \ge 0$ \cite{Sally}.
\end{itemize}

In principle the knowledge on $e(R) $ and $e_1(R)$ gives only partial information on the Hilbert polynomial and asymptotic information on the Hilbert function, nevertheless equalities in $(A)$ and $(N)$ force $G$ to be Cohen-Macaulay and to have a specific Hilbert function. If $G$ is also assumed to be Cohen-Macaulay, then $h_i \geq 0$ for all $i \geq 0$. 

Some of these inequalities are known to be true also with conditions different from $R$ being Cohen-Macaulay. For instance, if $G$ is a graded domain over an algebraically closed field, then (A) is known to hold (for instance, see \cite[Proposition 5.3]{GSyz}). This is not true in general even for local domains, as we will show later. Note that, if $R\cong Q/I$ for some regular local ring $(Q,\mm_Q)$ and $I \subseteq \mm_Q^2$, then $\delta(R) = \height_Q(I)$, and Abhyankar's inequality $(A)$ can be rewritten as $\e(R) \geq \height_Q(I)+1$. Similarly, if we write $G \cong P/J$ for $J \subseteq P=k[x_1,\ldots,x_n]$ a homogeneous ideal containing no linear forms, then $\delta(R)=\height_P(J)$ and the inequality becomes $\e(R)=\e(G) \geq \height_P(J)+1$.

Now we show that a variation of (A) holds even just assuming that $R$ is connected in codimension 1; for instance, if it is a domain.


\begin{theorem} \label{thm h2 and multiplicity linear forms}
Let $(R,\mm)$ be a complete local ring of dimension $d$, with algebraically closed residue field $k$. Write $G\cong P/J$ for some homogeneous ideal $J \subseteq P=k[x_1,\ldots,x_n]$ containing no linear forms, and set $\ell=\dim_k([\sqrt{J}]_1)$. If $R$ is connected in codimension one, then 
\begin{itemize} 
\item $\e(R) \geq \delta(R)+1-\ell$,
\item $h_2(R) \geq -d\ell.$ 
\end{itemize}
Moreover, if $G$ satisfies Serre's condition $(S_1)$ and $\ell \ne 0$, then the first inequality is strict.
\end{theorem}
\begin{proof}
The condition that $J$ contains no linear forms gives is equivalent to the fact that $\height_P(J) = \delta(R)$. If $\ell = 0$ the result follows at once from \cite[Theorem 5.9]{DMV}, because $G$ is connected in codimension $1$ by Theorem \ref{t:connectedness}. So let us assume $\ell>0$ henceforth.

If we set $\height_P(J)=:h$, the condition that $J_1=0$ gives
\[
h_2(R) = h_2(G) = \binom{h+1}{2} - \dim_k(J_2).
\]
Note that $h_1(G_{\rm red}) = h-\ell$. If we let $L=\sqrt{J}$, we obtain that
\[
h_2(G_{\rm red}) = \binom{h+1}{2} - \dim_k(L_2)+ d\ell.
\]
Since $J \subseteq L$, we have that
\[
h_2(G) +d\ell = \binom{h+1}{2} - \dim_k(J_2) + d\ell \geq \binom{h+1}{2} - \dim_k(L_2) + d\ell = h_2(G_{\rm red}).
\]
As $h_2(G_{\rm red})$ does not depend on the given presentation as a quotient of $P$, we may kill the linear forms contained in $L_1$ and assume that the ideal defining $G_{\rm red}$ contains no linear forms. By \cite[Theorem 5.9]{DMV} we conclude that $h_2(G_{\rm red}) \geq 0$, and the second claimed inequality follows. For the first, note that $\e(R) = \e(G) \geq \e(G_{{\rm red}})$. As above, if we choose a presentation of $G_{{\rm red}}$ for which the defining ideal does not contain linear forms, then by our assumption its height is $h-\ell$. The inequality now follows again from \cite[Theorem 5.9]{DMV} (see also \cite{EisenbudGoto}).

Finally, assume that $G$ satisfies Serre's condition $(S_1)$. Then every associated prime of $G$ is minimal, and since $G$ is connected in codimension $1$ it is equidimensional, hence unmixed. From the associativity formula for multiplicities, we have that 
\[
\e(G) = \sum_{\pp \in \Ass(G)} \e(G/\pp) \ell_{G_\pp}(G_\pp) \geq \sum_{\pp\in \Ass(G)} \e(G/\pp) = \e(G_{{\rm red}}).
\]
To conclude it suffices to show that there exists $\pp \in \Ass(G)$ such that $\ell_{G_\pp}(G_\pp) >1$. Since $L_1 \ne 0$ but $J_1 = 0$, and using that all associated primes of $G=P/J$ are minimal, we conclude that there must exist a primary component $\qq$ of $J$ such that $\qq_1 = 0$ and $\pp = \sqrt{\qq}$ contains a linear form. Let $x \in \pp_1 \smallsetminus \qq_1$. Since $\qq$ is $\pp$-primary, we also have $x \notin \qq G_\pp$, that is, the inclusions $(0) \subsetneq (x)G_\pp \subsetneq G_\pp$ are strict. It follows that $\ell_{G_\pp}(G_\pp) \geq 2$, as desired.
\end{proof}

In \cite{ST}, Srinivas and Trivedi produce a family of ideals $\{I_n\}_{n >0}$ inside $Q=k\ps{x,y,z,w}$ ($k$ is any field) such that $Q/I_n$ is $2$-dimensional, $\e(Q/I_n) = 4$, 
and that for odd values of $n$ is an integral domain. In particular, note that for $n \gg 0$ the inequalities (N) and (S) fail; however, the members of this family still satisfy (A) since $\height_Q(I_n) = 2$ for all $n$. Moreover, Srinivas and Trivedi show that $\HP_{Q/I_n} \ne \HP_{Q/I_{n'}}$ if $n \ne n'$. This is in stark contrast with the graded setup, where Kleiman \cite{Kleiman} proved that, for a graded domain of fixed dimension and multiplicity, there are only finitely many possible Hilbert functions.

A surprising result in \cite{GGHOPV} characterizes the Cohen-Macaulayness of an unmixed local ring $R$ in terms of the vanishing of the $e_1(\qq) $ where $\qq$ is a minimal reduction of $\mm$.  In particular $\e_1(\qq) \leq 0$ always holds, with equality if and only if $R$ is Cohen-Macaulay. An analogous investigation for the property of being Buchsbaum was discussed by Goto and Ozeki in \cite{GotoOzeki}. The integer $e_1(\qq)$ was considered by Goto and Nishida as a correction term in Northcott's inequality (N), in order to get rid of the assumption that $R$ is Cohen-Macaulay:
\begin{itemize}
\item (GN) $\e_1(R) - \e_1(\qq) \geq \e(R) -1$   \cite{GotoNishida}.
\end{itemize}

In the same setup as \cite{ST}, we produce a family of prime ideals $\{\pp_n\}_{n >0}$ that still fail to satisfy (N) and (S), but also do not satisfy (A). In view of Theorem \ref{thm h2 and multiplicity linear forms}, the radical of their initial ideal must contain linear forms. We also prove that  the family of local domains we present satisfies the equality in (GN).

\begin{example} \label{Ex}
Let $k$ be any field, $Q=k\ps{x,y,z,w}$ and $n \geq 1$ be an integer. Consider the family of ideals
\[
\pp_n=(x^2-z^{2n+1}w,xy-z^{n+1}w^{n+1},y^2-zw^{2n+1},yz^n-xw^n),
\]
and let $G_n = P/\init(\pp_n)$ be the associated graded ring of $Q/\pp_n$, where $P=k[X,Y,Z,W]$. We claim that for every $n \geq 1$:
\begin{enumerate}
\item $Q/\pp_n$ is a $2$-dimensional domain with an isolated singularity. Moreover, $Q/\pp_n \cong k\ps{st^{2n+1},s^{2n+1}t,t^2,s^2}$.
\item $\e(Q/\pp_n) = 2$, and $\e_1(Q/\pp_n) = 1-n$. In particular, $\e(Q/\pp_n)<\height_Q(\pp_n)+1$.
\item $\init(\pp_n) = (X^2,XY,Y^2,YZ^n-XW^n)$ is $(X,Y)$-primary, and $G_n$ satisfies Serre's condition $(S_1)$.
\item If $\HP_n(t)$ denotes the Hilbert polynomial of $G_n$, then $\HP_n \ne \HP_{n'}$ if $n \ne n'$.
\end{enumerate}
Let $n \geq 1$. Let $\qq_n$ be the kernel of the $k$-algebra homomorphism $\varphi_n:Q\to k\ps{s,t}$ sending
\[
x \mapsto st^{2n+1} \quad y \mapsto s^{2n+1}t \quad z \mapsto t^2 \quad w \mapsto s^2.
\]
One can directly check that $\pp_n \subseteq \qq_n$. Since $\qq_n$ is a prime ideal of height $2$, and the generators of $\pp_n$ have no common factor, we must have $\height(\pp_n)=2$. In particular, $\qq_n$ is a minimal prime of $\pp_n$. Now set $J_n=(X^2,XY,Y^2,YZ^n-XW^n)$. If we invert $Z$, then $(J_n)_Z=(X^2,Y-XW^nZ^{-n})_Z$ is a complete intersection. Similarly for $W$. Thus, if $J_n$ had any embedded associated prime, it would have to contain $\sqrt{J_n}+(Z,W) = (X,Y,Z,W)$.  However, an easy computation shows that $\init_{{\rm degrevlex}}(J_n) = (X^2,XY,Y^2,YZ^n)$, which has positive depth. It follows that $(X,Y,Z,W)$ cannot be associated to $P/J_n$, and thus $J_n$ is unmixed. 
Since the multiplicity does not change by passing to the initial ideal $ \init_{{\rm degrevlex}}(J_n) = (X^2,Y) \cap (X,Y^2,Z^n)$, we have that $\e(P/J_n) = \e(P/(X^2,Y)) = 2$. Note that $J_n$ is not radical but it is unmixed, and because $\e(P/J_n) = 2$, we conclude that $J$ is a $(X,Y)$-primary ideal. In particular, since $1<\e(P/\init(\pp_n)) \leq \e(P/J_n)=2$, we conclude that $J_n$ and $\init(\pp_n)$ have the same multiplicity, so that $\e(Q/\pp_n) = \e(G_n) = 2$. Note that, as $J_n$ is unmixed, the equality on multiplicities actually forces $J_n=\init(\pp_n)$. This also shows that $G_n = P/J_n$ satisfies Serre's condition $(S_1)$, and (3) is proved. For (2), using that the Hilbert Series of $G_n$ does not change by passing to the initial ideal $(X^2,XY,Y^2,YZ^n)$, an easy calculation shows that 
\[
{\rm HS}_{G_n}(t) = \frac{h(t)}{(1-t)^2}, \text{ where } h(t) = 1+2t-t^{n+1}.
\]
It follows that $\e_1(G_n) = h'(1) = 1-n$, and (2) is proved. Moreover, since $\e_2(G_n) = -n(n+1)$, the Hilbert polynomial of $G_n$ is 
\[
\HP_n(i) = \e_0(G_n)\binom{i+2}{2} - \e_1(G_n)(i+1) + \e_2(G_n) = i^2 + (n+2)i - \frac{(n+1)(n-2)}{2},
\]
and (3) follows. Finally, for (1), we consider the Jacobian matrix of $Q/\pp_n$:
\[
\begin{bmatrix} 2x & y & 0 & -w^n \\ 0 & x & 2y & z^n \\ -(2n+1)z^{2n}w & -(n+1)z^nw^{n+1} & -w^{2n+1} & -nyz^{n-1} \\ -z^{2n+1} & -(n+1)z^{n+1}w^n & -(2n+1)zw^n & -nxw^{n-1}
\end{bmatrix}.
\]
The $2$-minor corresponding to rows and columns $[24 | 14]$ is $z^{3n+1}$, and the $2$-minor corresponding to $[13|34]$ is $-w^{3n+1}$. Thus both $z$ and $w$ are in the radical of the Jacobian ideal of $Q/\pp_n$, and since they form a full system of parameters it follows that $Q/\pp_n$ has an isolated singularity. Since $G_n$ is unmixed, so is $Q/\pp_n$, and it follows that $\pp_n$ is a prime ideal. In particular, $\pp_n=\qq_n$, and thus $Q/\pp_n \cong k\ps{st^{2n+1},s^{2n+1}t,t^2,s^2}$.

We now make some further remarks on the family of examples we constructed. We let $R_n = Q/\pp_n$, with maximal ideal $\mm_n$. Firstly, we observe that if we take a minimal free resolution of $R_n$ over $Q$
\[
\xymatrix{
0 \ar[r] & Q \ar[rr]^-{\begin{bmatrix} w^n \\ z^n \\ -y \\ -x \end{bmatrix}} && Q^4 \ar[rrrrr]^-{\begin{bmatrix} y & 0 & w^n & 0 \\ -x & y & -z^n & w^n \\ 0 & -x & 0 & -z^n \\ z^{n+1}w & zw^{n+1} & x & y \end{bmatrix}} &&&&& Q^4 \ar[r] & Q \ar[r] & 0,
}
\]
then by local duality it follows that
\[
\ell_{R_n}(H^1_{\mm_n}(R_n)) = \ell_Q(\Ext^3_Q(R_n,Q)) = \ell_Q(Q/(x,y,z^n,w^n)) = n^2.
\]
We remark that the length of $H^1_{\mm_n}(R_n)$ must be an unbounded function of $n$ in view of Example \ref{Ex}, and in accordance with \cite{Trivedi_generalized}.

Lastly, regarding the inequality (GN), we show that the family of examples constructed above is extremal in the sense that it gives equalities in the bound of Goto and Nishida.

Consider the parameter ideal $\qq=(z,w)$. Since $z$ is a superficial element, we have that $\e_1(\qq) = \e_1(\overline{\qq})$, where $\overline{\qq} = (w)$ is the image of $J$ inside $\overline{R_n} = R_n/(z)$. Since $\qq$ is a minimal reduction of $\mm_n$, we have that $\e_0(\qq) = \e_0(\overline{\qq}) = \e(R_n) = 2$. A direct calculation shows that for $j \gg 0$ one has
\[
\ell_{\overline{R_n}}\left(\frac{\overline{R_n}}{(w^{j+1})}\right) = \ell_Q\left(\frac{k\ps{x,y,w}}{(x^2,xy,y^2,xw^n,w^{j+1})}\right) = 2(j+1) + n,
\]
and it follows that $\e_1(\qq) = -n$. Since $\e_1(R_n) =1-n$ and $\e(R_n)=2$, it follows that $\e_1(R_n) - \e_1(\qq) = \e(R_n) -1$, as claimed.
\end{example}

\begin{remark} In order to produce examples of complete local domains that do not satisfy (A), the family we exhibit in Example \ref{Ex} is ``minimal'' in many respects: (A) holds true for one-dimensional domains as they are Cohen-Macaulay. The assumption that the height is at least two is clear, as well as the assumption that the multiplicity is at least two. Moreover, in view of Theorem \ref{thm h2 and multiplicity linear forms} and the fact that $G_n$ as in the example has no embedded associated primes, the condition that $\sqrt{\init(\pp_n)}$ contains more than one linear form is also necessary. Finally, observe that $G_n$ is of minimal degree in the sense of \cite[Section 3]{EisenbudGoto}, as one can check that $\dim_k\left(\left[H^1_{(x,y,z,w)}(G_n)\right]_{0}\right)=1$ for all $n$.
\end{remark}

The family of prime ideals in Example \ref{Ex} shows, in particular, that infinitely many polynomials $h(t)\in\ZZ[t]$ can arise as the numerator of the Hilbert series of a local domain of fixed multiplicity. This cannot happen if the associated graded ring is reduced:

\begin{theorem}\label{tGred}
Fix an integer $e>0$. There exist constants, depending only on $e$, bounding respectively the projective dimension and the Betti numbers of any local ring $R=Q/I$ of multiplicity $e$, connected in codimension $1$, with algebraically closed residue field, and whose associated graded ring is reduced; here $(Q,\mm_Q)$ is a regular local ring, possibly of mixed characteristic and ramified, such that $I \subseteq \mm_Q^2$ (projective dimension and Betti numbers of $R$ are intended over $Q$). Also, the set of polynomials $h(t)\in\ZZ[t]$ such that $h(t)/(1-t)^{\dim(R)}$ is the Hilbert series of a local ring $R$ connected in codimension $1$, with algebraically closed residue field, of multiplicity $e$, and such that its associated graded ring is reduced, is finite.
\end{theorem}
\begin{proof}
Let $R=Q/I$ as in the statement, and let $k$ be its algebraically closed residue field. Let $P={\rm gr}_{\mm_Q}(Q) \cong k[x_1,\ldots ,x_n]$, where $n$ is also the embedding dimension of $R$ because of the assumption $I \subseteq \mm_Q^2$. In particular, $G\cong P/J$ with $J=\init(I) \subseteq P$ a homogeneous ideal containing no linear forms. We have that $G$ is connected in codimension $1$, and hence equidimensional, by Theorem \ref{t:connectedness}. Furthermore, $J$ is a radical ideal not containing any linear forms. These two conditions, together with the fact that $k$ is algebraically closed, imply that $\height(J)<e$ (e.g. see \cite[Proposition 5.2]{GSyz}). By \cite[Theorem 5.2]{GCMPV}, there exist constants, depending only on $e$, bounding respectively the projective dimension and the Betti numbers (over $P$) of any such $G=P/J$. Since projective dimension and Betti numbers can only increase when passing to the associated graded ring (for instance, see \cite{Robbiano,RossiSharifan,Sammartano}), the first part of the statement follows.

For the  part concerning $h$-polynomials, again by \cite[Theorem 5.2]{GCMPV} there exist constants, depending only on $e$, bounding respectively the projective dimension, the Castelnuovo-Mumford regularity and the graded Betti numbers (over $P$) of any $G=P/J$ arising as the associated ring of a local ring $R$ as in the statement of the theorem. We conclude because
\[HS_G(t)=\frac{(1-t)^{-\height(J)}\sum_{i=0}^{\mathrm{pdim}(G)}(-1)^i\sum_{j=0}^{\mathrm{reg}(G)}\beta_{i,i+j}t^j}{(1-t)^{\dim(R)}}\]
and, as already pointed out, $\height(J)<e$. Thus there are only finitely many possibilities for the numerator $h_G(t)=h_R(t)$ of the Hilbert series of $R$.
\end{proof}

\begin{remark}
While the finiteness of the set of $h$-polynomials in Theorem \ref{tGred}, even fixing the dimension, fails by \cite{ST}, or also in view of Example \ref{Ex} without the assumption that the associated graded ring is reduced, we do not know examples showing that this assumption is necessary for the boundedness of the projective dimension and the Betti numbers. We therefore propose the following question: is there a bound in terms of the Hilbert-Samuel multiplicity of a complete local domain $R=Q/I$, with algebraically closed residue field, for the minimal number of generators of $I$? And for the projective dimension of $R$ over $Q$? In these questions $(Q,\mm_Q)$ is a regular local ring such that $I \subseteq \mm_Q^2$. As a variant of the questions above one can also fix the Krull dimension of $R$ in addition to the multiplicity. 
\end{remark}

\section*{Acknowledgments} We thank Craig Huneke, Hailong Dao, and Linquan Ma for several useful discussions. We thank Linquan Ma and Kevin Tucker for suggesting a shorter and more conceptual proof of Lemma \ref{l:tricky general}. We would also like to thank the anonymous referee for useful comments. The authors were partially supported by the PRIN 2020 project 2020355B8Y ``Squarefree Gr{\"o}bner degenerations, special varieties and related topics'', by the MIUR Excellence Department Project CUP D33C23001110001, and by INdAM-GNSAGA. This material is based upon work supported by the National Science Foundation under Grant No. DMS-1928930 and by the Alfred P. Sloan Foundation under grant G-2021-16778, while the first and third authors were in residence at the Simons Laufer Mathematical Sciences Institute (formerly MSRI) in Berkeley, California, during the Spring 2024 semester.

\section*{Data Availability Statement}
Data sharing not applicable to this article as no datasets were generated or analysed during the current study.

\section*{Conflict of Interest Statement}
On behalf of all authors, the corresponding author states that there is no conflict of interest. 

\bibliographystyle{alpha}
\bibliography{References}

\newcommand{\etalchar}[1]{$^{#1}$}
\def\cprime{$'$} \def\cprime{$'$}
\begin{thebibliography}{ALNBRM22}

\bibitem[AA82]{AvramovAchilles}
R\"{u}diger Achilles and Luchezar~L. Avramov.
\newblock Relations between properties of a ring and of its associated graded
  ring.
\newblock In {\em Seminar {D}. {E}isenbud/{B}. {S}ingh/{W}. {V}ogel, {V}ol. 2},
  volume~48 of {\em Teubner-Texte Math.}, pages 5---29. Teubner, Leipzig, 1982.

\bibitem[Abh67]{Abhyankar}
Shreeram~Shankar Abhyankar.
\newblock Local rings of high embedding dimension.
\newblock {\em Amer. J. Math.}, 89:1073--1077, 1967.

\bibitem[ALNBRM22]{ALNBRM}
Lilia Alan\'{\i}s-L{\'o}pez, Luis N{\'u}{\~n}ez-Betancourt, and Pedro
  Ram\'{\i}rez-Moreno.
\newblock Connectedness of square-free {G}roebner deformations.
\newblock {\em Proc. Amer. Math. Soc.}, 150(4):1405--1419, 2022.

\bibitem[BCRV22]{BCRV}
Winfried Bruns, Aldo Conca, Claudiu Raicu, and Matteo Varbaro.
\newblock {\em Determinants, Gr\"obner basis and Cohomology}.
\newblock Springer Monographs in Mathematics. Springer Cham, 2022.

\bibitem[BR86]{BrodmannRung}
Markus Brodmann and Josef Rung.
\newblock Local cohomology and the connectedness dimension in algebraic
  varieties.
\newblock {\em Comment. Math. Helv.}, 61(3):481--490, 1986.

\bibitem[Bro86]{Brodmann}
M.~Brodmann.
\newblock A few remarks on blowing-up and connectedness.
\newblock {\em J. Reine Angew. Math.}, 370:52--60, 1986.

\bibitem[BS87]{BaSt}
David Bayer and Michael Stillman.
\newblock A criterion for detecting {$m$}-regularity.
\newblock {\em Invent. Math.}, 87(1):1--11, 1987.

\bibitem[BS13]{BroSharp}
M.~P. Brodmann and R.~Y. Sharp.
\newblock {\em Local cohomology}, volume 136 of {\em Cambridge Studies in
  Advanced Mathematics}.
\newblock Cambridge University Press, Cambridge, second edition, 2013.
\newblock An algebraic introduction with geometric applications.

\bibitem[CCM{\etalchar{+}}19]{GCMPV}
Giulio Caviglia, Marc Chardin, Jason McCullough, Irena Peeva, and Matteo
  Varbaro.
\newblock Regularity of prime ideals.
\newblock {\em Math. Z.}, 291(1-2):421--435, 2019.

\bibitem[CN81]{CavaliereNiesi}
Maria~Pia Cavaliere and Gianfranco Niesi.
\newblock On {S}erre's conditions in the form ring of an ideal.
\newblock {\em J. Math. Kyoto Univ.}, 21(3):537--546, 1981.

\bibitem[Coh46]{Cohen}
I.~S. Cohen.
\newblock On the structure and ideal theory of complete local rings.
\newblock {\em Trans. Amer. Math. Soc.}, 59:54--106, 1946.

\bibitem[CV20]{ConcaVarbaro}
Aldo Conca and Matteo Varbaro.
\newblock Square-free {G}r\"{o}bner degenerations.
\newblock {\em Invent. Math.}, 221(3):713--730, 2020.

\bibitem[DCEP82]{DEP}
Corrado De~Concini, David Eisenbud, and Claudio Procesi.
\newblock {\em Hodge algebras}, volume~91 of {\em Ast\'{e}risque}.
\newblock Soci\'{e}t\'{e} Math\'{e}matique de France, Paris, 1982.
\newblock With a French summary.

\bibitem[DMV24]{DMV}
Hailong Dao, Linquan Ma, and Matteo Varbaro.
\newblock Regularity, singularities and {$h$}-vector of graded algebras.
\newblock {\em Trans. Amer. Math. Soc.}, 377(3):2149--2167, 2024.

\bibitem[EG84]{EisenbudGoto}
David Eisenbud and Shiro Goto.
\newblock Linear free resolutions and minimal multiplicity.
\newblock {\em J. Algebra}, 88(1):89--133, 1984.

\bibitem[Eis05]{GSyz}
David Eisenbud.
\newblock {\em The geometry of syzygies}, volume 229 of {\em Graduate Texts in
  Mathematics}.
\newblock Springer-Verlag, New York, 2005.
\newblock A second course in commutative algebra and algebraic geometry.

\bibitem[Fr{\"o}87]{Froberg}
Ralf Fr{\"o}berg.
\newblock Connections between a local ring and its associated graded ring.
\newblock {\em J. Algebra}, 111(2):300--305, 1987.

\bibitem[GGH{\etalchar{+}}10]{GGHOPV}
L.~Ghezzi, S.~Goto, J.~Hong, K.~Ozeki, T.~T. Phuong, and W.~V. Vasconcelos.
\newblock Cohen-{M}acaulayness versus the vanishing of the first {H}ilbert
  coefficient of parameter ideals.
\newblock {\em J. Lond. Math. Soc. (2)}, 81(3):679--695, 2010.

\bibitem[GN94]{GotoNishida}
Shiro Goto and Koji Nishida.
\newblock Filtrations and the {G}orenstein property of the associated {R}ees
  algebras.
\newblock {\em Mem. Amer. Math. Soc.}, 110(526):69--134, 1994.

\bibitem[GO10]{GotoOzeki}
Shiro Goto and Kazuho Ozeki.
\newblock Buchsbaumness in local rings possessing constant first {H}ilbert
  coefficients of parameters.
\newblock {\em Nagoya Math. J.}, 199:95--105, 2010.

\bibitem[Gre98]{Green}
Mark~L. Green.
\newblock Generic initial ideals.
\newblock In {\em Six lectures on commutative algebra ({B}ellaterra, 1996)},
  volume 166 of {\em Progr. Math.}, pages 119--186. Birkh\"{a}user, Basel,
  1998.

\bibitem[Gro05]{SGA2}
Alexander Grothendieck.
\newblock {\em Cohomologie locale des faisceaux coh\'{e}rents et
  th\'{e}or\`emes de {L}efschetz locaux et globaux ({SGA} 2)}, volume~4 of {\em
  Documents Math\'{e}matiques (Paris) [Mathematical Documents (Paris)]}.
\newblock Soci\'{e}t\'{e} Math\'{e}matique de France, Paris, 2005.
\newblock S\'{e}minaire de G\'{e}om\'{e}trie Alg\'{e}brique du Bois Marie,
  1962, Augment\'{e} d'un expos\'{e} de Mich\`ele Raynaud. [With an expos\'{e}
  by Mich\`ele Raynaud], With a preface and edited by Yves Laszlo, Revised
  reprint of the 1968 French original.

\bibitem[Har62]{HartCICo}
Robin Hartshorne.
\newblock Complete intersections and connectedness.
\newblock {\em Amer. J. Math.}, 84:497--508, 1962.

\bibitem[Kle71]{Kleiman}
Steven Kleiman.
\newblock {\em Exp XIII in A. Grothendieck et al., Th\'{e}orie des
  intersections et th\'{e}or\`eme de {R}iemann-{R}och}, volume Vol. 225 of {\em
  Lecture Notes in Mathematics}.
\newblock Springer-Verlag, Berlin-New York, 1971.
\newblock S\'{e}minaire de G\'{e}om\'{e}trie Alg\'{e}brique du Bois-Marie
  1966--1967 (SGA 6), Dirig\'{e} par P. Berthelot, A. Grothendieck et L.
  Illusie. Avec la collaboration de D. Ferrand, J. P. Jouanolou, O. Jussila, S.
  Kleiman, M. Raynaud et J. P. Serre.

\bibitem[KM05]{KnMi}
Allen Knutson and Ezra Miller.
\newblock Gr\"{o}bner geometry of {S}chubert polynomials.
\newblock {\em Ann. of Math. (2)}, 161(3):1245--1318, 2005.

\bibitem[KS95]{KalkSturmfels}
Michael Kalkbrener and Bernd Sturmfels.
\newblock Initial complexes of prime ideals.
\newblock {\em Adv. Math.}, 116(2):365--376, 1995.

\bibitem[KW88]{KW}
Heinz Kredel and Volker Weispfenning.
\newblock Computing dimension and independent sets for polynomial ideals.
\newblock {\em J. Symbolic Comput.}, 6(2-3):231--247, 1988.
\newblock Computational aspects of commutative algebra.

\bibitem[Mat89]{MatsumuraRing}
Hideyuki Matsumura.
\newblock {\em Commutative ring theory}, volume~8 of {\em Cambridge Studies in
  Advanced Mathematics}.
\newblock Cambridge University Press, Cambridge, second edition, 1989.
\newblock Translated from the Japanese by M. Reid.

\bibitem[Nor60]{Northcott}
D.~G. Northcott.
\newblock A note on the coefficients of the abstract {H}ilbert function.
\newblock {\em J. London Math. Soc.}, 35:209--214, 1960.

\bibitem[Rat78]{Ratliff}
Louis~J. Ratliff, Jr.
\newblock {\em Chain conjectures in ring theory}, volume 647 of {\em Lecture
  Notes in Mathematics}.
\newblock Springer, Berlin, 1978.
\newblock An exposition of conjectures on catenary chains.

\bibitem[Rob81]{Robbiano}
Lorenzo Robbiano.
\newblock Coni tangenti a singolarit{\`a} razionali.
\newblock {\em Curve algebriche, Istituto di Analisi Globale, Firenze}, 1981.

\bibitem[RS10]{RossiSharifan}
Maria~Evelina Rossi and Leila Sharifan.
\newblock Consecutive cancellations in {B}etti numbers of local rings.
\newblock {\em Proc. Amer. Math. Soc.}, 138(1):61--73, 2010.

\bibitem[Sal92]{Sally}
Judith~D. Sally.
\newblock Hilbert coefficients and reduction number {$2$}.
\newblock {\em J. Algebraic Geom.}, 1(2):325--333, 1992.

\bibitem[Sam16]{Sammartano}
Alessio Sammartano.
\newblock Consecutive cancellations in {T}or modules over local rings.
\newblock {\em J. Pure Appl. Algebra}, 220(12):3861--3865, 2016.

\bibitem[ST97]{ST}
V.~Srinivas and Vijaylaxmi Trivedi.
\newblock A finiteness theorem for the {H}ilbert functions of complete
  intersection local rings.
\newblock {\em Math. Z.}, 225(4):543--558, 1997.

\bibitem[{Sta}18]{Stacks}
The {Stacks Project Authors}.
\newblock \textit{Stacks Project}.
\newblock \url{https://stacks.math.columbia.edu}, 2018.

\bibitem[Tei36]{Teichmuller}
Oswald Teichm{\"u}ller.
\newblock {\"U}ber die {S}truktur diskret bewerteter perfekter {K}{\"o}rper.
\newblock {\em Nachrichten aus der Mathematik}, 1(10):152--161, 1936.

\bibitem[Tri01]{Trivedi_generalized}
Vijaylaxmi Trivedi.
\newblock Finiteness of {H}ilbert functions for generalized {C}ohen-{M}acaulay
  modules.
\newblock {\em Comm. Algebra}, 29(2):805--813, 2001.

\bibitem[Var09]{Varbaro}
Matteo Varbaro.
\newblock Gr\"{o}bner deformations, connectedness and cohomological dimension.
\newblock {\em J. Algebra}, 322(7):2492--2507, 2009.

\end{thebibliography}

\end{document}